\documentclass{amsart}

\usepackage{enumerate}
\usepackage{amsthm}
\usepackage{amssymb}
\usepackage{color}
\usepackage{hyperref}

\numberwithin{equation}{section}

\newtheorem{theorem}{Theorem}[section]
\newtheorem{lemma}[theorem]{Lemma}

\theoremstyle{definition} 
\theoremstyle{definition} 

\DeclareMathOperator{\Tr}{Tr}

\newcommand{\R}{{\mathbf{R}}}
\newcommand{\E}{{\mathbf{E}}}
\newcommand{\N}{{\mathbf{N}}}
\newcommand{\D}{{\mathcal{D}}}
\newcommand{\dom}{{\mathrm{dom}}}

\newcommand{\LB}{{\mathcal{L}}}
\newcommand{\half}{{\frac{1}{2}}}

\newcommand{\diff}[1]{\,\mathrm{d}#1}

\newcommand{\g}{g}
\newcommand{\Cb}{\mathcal{C}_{\mathrm{b}}}
\definecolor{darkred}{rgb}{.6,0,0}

\title[Weak convergence for the nonlinear stochastic heat equation]{Weak convergence for a spatial approximation of the nonlinear stochastic heat equation}

\author{Adam Andersson}
\address{
  Department of Mathematical Sciences,
  Chalmers University of Technology and University of Gothenburg,
  SE--412 96 Gothenburg,
  Sweden}

\email{adam.andersson@chalmers.se}

\author{Stig Larsson}
\address{
  Department of Mathematical Sciences,
  Chalmers University of Technology and University of Gothenburg,
  SE--412 96 Gothenburg,
  Sweden}

\email{stig@chalmers.se}

\begin{document}

\begin{abstract} We find the weak rate of convergence of the spatially semidiscrete finite element approximation of the nonlinear stochastic heat equation.  Both multiplicative and additive noise is considered under different assumptions.  This extends an earlier result of Debussche in which time discretization is considered for the stochastic heat equation perturbed by white noise. It is known that this equation has a solution only in one space dimension. In order to obtain results for higher dimensions, colored noise is considered here, besides white noise in one dimension. Integration by parts in the Malliavin sense is used in the proof. The rate of weak convergence is, as expected, essentially twice the rate of strong convergence.  \end{abstract}



\maketitle

\section{Introduction and main result}
Let $\D\subset\R^d$, $d=1,2,3$, be a bounded, convex, polygonal domain. We consider, for $T>0$, the stochastic heat equation with Dirichlet boundary condition, written in abstract form as a stochastic evolution equation in $H=L_2(\D)$:
\begin{equation}
\label{SPDE}
\diff{X(t)}+[AX(t)-f(X(t))]\diff{t} = \g(X(t))\diff{W(t)},\; t\in(0,T];\quad X(0)=X_0.
\end{equation}
This equation is driven by a $Q$-Wiener process $(W(t))_{t\in[0,T]}$ with respect to a filtered probability space $(\Omega, \mathcal{F},(\mathcal{F}_t)_{t\in[0,T]},\mathbf{P})$. The covariance operator $Q$ is selfadjoint and positive semidefinite, not necessarily of finite trace. For technical reasons we consider a deterministic initial value $X_0\in H$.

The leading linear operator $A$ is, for simplicity, taken to be $-\Delta$ with domain $\dom(A)=H^2(\D)\cap H_0^1(\D)$, where $\Delta=\sum_{k=1}^d \partial^2/\partial x_k^2$ is the Laplace operator. It is well known that $-A$ generates an analytic semigroup of bounded linear operators on $H$. We denote it by $(S(t))_{t\geq0}$.  The spaces $\dot{H}^\beta=\dom(A^{\frac{\beta}2})$, defined by fractional powers of $A$, are used to measure the spatial regularity.  We denote the norm and inner product in $H=L_2(\D)$ by $\|\cdot\|$ and $\langle\cdot,\cdot\rangle$.

Let $U,V$ be separable Hilbert spaces and let $\LB(U,V)$ denote the Banach space of all bounded linear operators. We denote by $\LB_2(U,V)\subset\LB(U,V)$ the subspace consisting of all Hilbert-Schmidt operators. We use the abbreviations $\LB(U)=\LB(U,U)$, $\LB=\LB(H)$, and similarly for $\LB_2(U)$ and $\LB_2$. For $T\in\LB(U)$, selfadjoint and positive semidefinite, we write $T\geq0$. By $\Cb^k(U,V)$ we denote the space of not necessarily bounded functions from a Banach space $U$ to a Banach space $V$ that have continuous and bounded Fr\'{e}chet derivatives of orders $1,\dots,k$. For more precise definitions, see Section \ref{sec:preliminaries} below.

Recall that $Q\in\LB(H)$, $Q\geq0$, $H=L_2(\D)$, $\LB=\LB(H)$, and $\LB_2=\LB_2(H)$. Let $U_0=Q^{\frac12}(H)$ and $\LB_2^0=\LB_2(U_0,H)$. We use a ``regularity parameter'' $\beta>0$ such that $\|A^{\frac{\beta-1}2}\|_{\LB_2^0}=\|A^{\frac{\beta-1}2}Q^\frac12\|_{\LB_2}<\infty$.  If $Q=I$, then $\|A^{\frac{\beta-1}2}\|_{\LB_2^0}=\|A^{\frac{\beta-1}2}\|_{\LB_2}<\infty$, if and only if $d=1$ and $\beta<\frac12$, see \eqref{Awhite}.  We consider two sets of assumptions according to the type of noise term.

\begin{itemize}
\item[\textbf{A.}] \emph{Additive noise in multiple dimensions.}  Assume that $f\in \Cb^2(H,H)$, $g(x)=I\in\LB$ (the identity operator) for all $x\in H$, and $\|A^{\frac{\beta-1}2}\|_{\LB_2^0} =\|A^{\frac{\beta-1}2}Q^\frac12\|_{\LB_2}<\infty$ for some $\beta\in[\frac12,1]$.

\item[\textbf{B.}] \emph{Multiplicative noise in one dimension.} Assume that $f\in \Cb^2(H,H)$, $g(x)=B+Cx+\tilde{g}(x)$, where $B\in\LB$, $C\in\LB(H,\LB)$, and $\tilde{g}\in\Cb^2(\dot{H}^{-\frac12},\LB)$.  Moreover, assume that $d=1$, $Q=I$, and select any $\beta\in(0,\frac12)$.  (Thus, $\|A^{\frac{\beta-1}2}\|_{\LB_2}<\infty$.)  \end{itemize}

Under either of these assumptions we have a unique mild solution to \eqref{SPDE}, i.e., a process $(X(t))_{t\in[0,T]}$ satisfying the stochastic fixed point equation
\begin{equation}
\label{mild}
X(t)=S(t)X_0+\int_0^tS(t-s)f(X(s))\diff{s}+\int_0^tS(t-s)\g(X(s))\diff{W(s)},
\quad t\in[0,T].
\end{equation}
One can also show that the solution has spatial regularity of order $\beta$, i.e., it is of the form $X\colon[0,T]\times\Omega\rightarrow\dot{H}^\beta$, $\mathbf{P}$-almost surely, see Theorem \ref{Existence} below and the discussion preceding it.

In this paper we consider spatial discretization of \eqref{SPDE} by means of a standard finite element method.  Let $(V_h)_{h\in(0,1)}$ be a family of spaces of continuous piecewise linear functions corresponding to a quasi-uniform family $(T_h)_{h\in(0,1)}$ of triangulations of $\D$ with $V_h\subset H_0^1(\D)$. The parameter $h$ is the mesh size of $T_h$. Quasi-uniformity is a rather weak assumption used in some parts of the finite element literature. It excludes highly graded meshes but allows us to use the inverse inequality $\|\varphi\|_{\dot{H}^1}\le Ch^{-1}\|\varphi\|$ for $\varphi\in V_h$ to simplify some arguments.
Let $P_h\colon H\rightarrow V_h$ denote the orthogonal projection. We define the discrete Laplacian to be the operator $A_h\colon V_h\rightarrow V_h$ satisfying %
\begin{equation}
\label{Ah}
\langle A_h\psi,\chi\rangle=\langle\nabla\psi,\nabla\chi\rangle,\quad\forall\psi,\chi\in V_h.
\end{equation}
The finite element approximation of the elliptic problem $Au=f$ is the unique solution of the equation $A_hu_h=P_hf$. It is known that $\|u_h-u\|=\|A_h^{-1}P_hf-A^{-1}f\|\leq Ch^2\|f\|$, if $f\in H$. The semigroup generated by $-A_h$ is denoted  $(S_h(t))_{t\ge0}$.
The spatially semidiscrete analogue of \eqref{SPDE} is to find a process $(X_h(t))_{t\in[0,T]}$ with values in $V_h$ such that
\begin{align}
  \label{eq:SPDEh}
\diff{X_h(t)}+[A_hX_h(t)-P_hf(X_h(t))]\diff{t} = P_h\g(X_h(t))\diff{W(t)},\; t\in(0,T];\quad X_h(0)=P_hX_0,
\end{align}
or in mild form,
\begin{align}
\label{SPDEapprox}
\begin{split}
X_h(t)&=S_h(t)P_hX_0+\int_0^tS_h(t-s)P_hf(X_h(s))\diff{s}
\\& \quad +\int_0^tS_h(t-s)P_h\g(X_h(s))\diff{W(s)},
\quad t\in[0,T].
\end{split}
\end{align}
The existence of a unique mild solution can be proved in a similar way as for \eqref{mild}. It is also known that $X_h(T)$ converges strongly to $X(T)$ with order $\beta$ under Assumptions \textbf{A} or \textbf{B}, see \eqref{eq:kruse-strong}. Our goal is to prove weak convergence in the form
\begin{equation*}
\E[\varphi(X(T))-\varphi(X_h(T))]=\mathcal{O}(h^{2\beta-\epsilon}), \quad \text{as $h\to 0$},
\end{equation*}
for any $\epsilon>0$ and any testfunction $\varphi\in \Cb^2(H,\R)$.

For an exhaustive list of references for approximations of stochastic partial differential equations, see, e.g., \cite{Debussche}. We mention some works related to the situation studied here. Weak convergence of numerical schemes for linear equations with additive noise is treated in \cite{DebusschePrintems}, \cite{LinearEq}, \cite{LinearEqII}, \cite{KruseThesis}, and \cite{LindnerSchilling}. In the first paper full discretization of the stochastic heat equation is considered for colored additive noise in multiple dimensions, i.e., our Assumption \textbf{A} with $f=0$. Papers \cite{LinearEq} and \cite{LinearEqII} deal with semidiscretization in space and full discretization, respectively, for the linear stochastic heat, Cahn-Hilliard, and wave equations, also with additive colored noise. In \cite{KruseThesis} a new method for proving weak convergence for linear equations based on Malliavin calculus is presented.  The paper \cite{LindnerSchilling} provides an extension to impulsive noise.

The only results on weak convergence for nonlinear equations are those of \cite{Brehier}, \cite{Schrodinger}, \cite{Debussche}, \cite{HausSemi}, \cite{HausWave}, and \cite{WangGan}. In \cite{HausSemi}, discretization in time with implicit Euler and Crank-Nicolson schemes is considered for semilinear parabolic equations with additive noise. Paper \cite{HausWave} treats the wave equation with additive white noise, discretized by a leap-frog scheme. This case is a bit different from the others, due to the lack of analyticity of the semigroup for the wave equation. In \cite{Schrodinger} semidiscretization in time for the nonlinear stochastic Schr\"{o}dinger equation with multiplicative white noise is considered.

The papers \cite{Schrodinger}, \cite{DebusschePrintems}, \cite{LinearEq}, \cite{LinearEqII}, and \cite{LindnerSchilling} express the weak error by means of a Kolmogorov equation after removing the linear term $AX(t)$ by a transformation of variables.  This transformation does not work for the nonlinear heat equation.  This difficulty is handled in \cite{Debussche} by means of an integration by parts from the Malliavin calculus.  This paper proves weak convergence of temporal semidiscretizations for the nonlinear heat equation with multiplicative noise in one space dimension, i.e., our Assumption \textbf{B}.  In \cite{Brehier} the method of \cite{Debussche} is used to prove weak convergence of the invariant measure of temporally discrete approximations under the same assumptions, except for an extra boundedness condition on the nonlinearity.  In \cite{WangGan} the same proof technique is used to study time discretization for the heat equation with additive noise in multiple dimensions, i.e., our Assumption \textbf{A}.

In the present paper we extend the results of \cite{Debussche} and \cite{WangGan} to spatial discretization.  Our Assumptions \textbf{B} and \textbf{A} coincide with the assumptions in these two papers, respectively.  Therefore we may quote some estimates from there.  One difficulty that arises in connection with the spatial discretization is that the projector $P_h$ does not commute with the projector onto eigenspaces of $A$.

In all these works the rate of weak convergence is, up to an arbitrary $\epsilon>0$, twice that of strong convergence. The Malliavin calculus is a useful tool in the study of weak convergence of semilinear equations. It has been utilized in \cite{Debussche}, \cite{HausSemi} and \cite{KruseThesis} in completely different ways. It plays a central role in the proof of our Theorem \ref{main}, following the method of \cite{Debussche}.

The result of this paper actually concerns the convergence of the law $\mathcal{L}(X_h(T))=\mathbf{P}\circ (X_h(T))^{-1}$ of the random variables $(X_h(T))_{h\in(0,1)}$, as the mesh size parameter $h\rightarrow0$. We say that the law of $X_h(T)$ converges weakly to that of $X(T)$, if $\E [\varphi(X_h(T))]\rightarrow \E [\varphi(X(T))]$ as $h\rightarrow0$, for all test functions $\varphi\in\Cb(H,\R)$, the space of all bounded continuous functions on $H$. This convergence follows under mild assumptions from the strong convergence $\E[\|X_h(T)- X(T)\|^2]=\mathcal{O}(h^\beta)$, see \cite{Kruse}, but with no better rate than $\beta$. For $\varphi\in\Cb^2(H,\R)$, we obtain in this paper the rate of weak convergence $2\beta-\epsilon$, for an arbitrary $\epsilon>0$.

\begin{theorem}
\label{main}
Assume either Assumption \textbf{A} or Assumption \textbf{B} and let $X$ and $X_h$ be the solutions of the equations \eqref{mild} and \eqref{SPDEapprox}, respectively. Then, for every test function $\varphi\in \Cb^2(H,\R)$ and $\gamma\in[0,\beta)$, we have the convergence
\begin{equation*}
|\E[\varphi(X(T))-\varphi(X_h(T))]|=\mathcal{O}(h^{2\gamma}),\quad\textrm{as }h\rightarrow0.
\end{equation*}
\end{theorem}

The weak error is interesting for various reasons. It measures the error made by sampling from an approximate probability law of $X(T)$, rather than the deviation from the trajectory of an exact solution, as for the strong error. The result tells us that the weak error, when approximating the quantity $\E[\varphi(X(T))]$ by $\E[\varphi(X_h(T))]$, is decreasing fast as $h\rightarrow0$ for smooth $\varphi$.

Section \ref{sec:preliminaries} is devoted to preliminaries. In Subsection \ref{subsec:compact} compact operators and tensor products are introduced. We need Schatten classes more general than the trace class and Hilbert-Schmidt operators. In Subsection \ref{subsec:frechet} some notation for Fr\'{e}chet derivatives is fixed. The semigroup framework and basic material on the finite element method are presented in Subsection \ref{subsec:functional}. In Subsection \ref{subsec:malliavin} the Malliavin calculus and stochastic integration is introduced. In Subsection \ref{subsec:existence} existence and uniqueness of the stochastic equations \eqref{mild} and \eqref{SPDEapprox} is stated. In Section \ref{sec:estimates} two moment estimates for the Malliavin derivative of $X_h(t)$ are proved. Section \ref{sec:kolmogorov} contains regularity results for the Kolmogorov equation, adapting results from \cite{Debussche} and \cite{WangGan} to our setting. The proof of Theorem \ref{main} is given in Section \ref{sec:main}.

\section{Preliminaries} \label{sec:preliminaries}
\subsection{Compact operators and tensor products} \label{subsec:compact} Given two separable real Hilbert spaces $(U,\langle\cdot,\cdot\rangle_U)$ and $(V,\langle\cdot,\cdot\rangle_V)$, let $\LB(U,V)$ denote the Banach space of all bounded linear operators $U\rightarrow V$ endowed with the uniform norm. We write $\LB(U)=\LB(U,U)$. Let $(\sigma_i)_{i\in\N}$ be the singular values of a compact operator $T\in\LB(U)$. These are the eigenvalues of the operator $|T|=(TT^*)^{1/2}$. The Schatten classes are the spaces:
\begin{equation*}
\LB_p(U)=\Big\{T\in\LB(U): \|T\|_{\LB_p(U)}=\Big(\sum_{i\in\N}\sigma_i^p\Big)^{\frac1p}<\infty\Big\} ,\;1\leq p<\infty;\quad \LB_\infty(U)=\LB(U).
\end{equation*}
They are Banach spaces. The class $\LB_1$ is the space of trace class operators. Take an arbitrary ON-basis $(e_n)_{n\in\N}\subset U$. We define the trace of an operator $T\in\LB_1(U)$ as the quantity
\begin{equation*}
\Tr(T)=\sum_{i\in\N}\langle Te_i,e_i\rangle_U.
\end{equation*}
It is independent of the choice of ON-basis. By $T\geq0$ we mean that $T$ is positive semidefinite. If $T\in\LB_1$, then
\begin{equation}
\label{TrL1}
|\Tr(T)|\leq\|T\|_{\LB_1}\quad\text{and}\quad \Tr(T)=\|T\|_{\LB_1}\textrm{ if } T\geq0.
\end{equation}
It follows directly from the definition that $\Tr(T)=\Tr(T^*)$ for $T\in\LB_1$. Moreover,
\begin{equation}
\label{TrST}
\Tr(ST)=\Tr(TS),
\end{equation}
whenever $S\in\LB(U,V)$ and $T\in\LB(V,U)$ satisfy $ST\in\LB_1(V)$ and $TS\in\LB_1(U)$.

More generally, the class $\LB_2(U,V)$ is the space of Hilbert-Schmidt operators from $U$ to $V$. It is defined as the Hilbert space with the scalar product and norm
\begin{align}
\label{defHS} \langle S,T\rangle_{\LB_2(U,V)}&=\sum_{i\in\N}\langle Se_i,Te_i\rangle_V=\Tr(T^*S)=\Tr(ST^*),\\
\label{defHS2}\|T\|_{\LB_2(U,V)} &=\Big(\sum_{i\in\N}\|Te_i\|_V^2\Big)^{\frac12}=\sqrt{\Tr(TT^*)}.
\end{align}
The choice of ON-basis $(e_n)_{n\in\N}\subset U$ is arbitrary. For $U=V$ the class $\LB_2=\LB_2(U)$ is alone to enjoy this property. For $\LB_p$ with $p\neq2$, only an eigenbasis of $|T|$ can be used.

The following H\"{o}lder type inequality for Schatten classes holds:
\begin{equation}
\label{Holder}
\|ST\|_{\LB_r}\leq\|S\|_{\LB_p}\|T\|_{\LB_q},\quad r^{-1}=p^{-1}+q^{-1},\quad p,q,r\in[1,\infty].
\end{equation}
The border case
\begin{equation}
\label{HolderInf}
\|ST\|_{\LB_r}\leq\|S\|_{\LB}\|T\|_{\LB_r}
\end{equation}
is included, meaning that $\LB_r(U)$ is an ideal of the Banach algebra $\LB(U)$. Also
\begin{equation}
\label{Holder1inf}
|\langle S,T\rangle_{\LB_2}|=|\Tr(ST^*)|\leq\|ST^*\|_{\LB_1}\leq\|S\|_{\LB}\|T\|_{\LB_1}.
\end{equation}
For more about the Schatten classes see \cite{DS}.

The tensor product space $U\otimes V$ of two Hilbert spaces $U$ and $V$ is a Hilbert space together with a bilinear mapping $U\times V\rightarrow U\otimes V, (u,v)\mapsto u\otimes v$ with dense range and with the inner product
\begin{equation*}
\langle u_1\otimes v_1,u_2\otimes v_2\rangle_{U\otimes V}=\langle u_1,u_2\rangle_U\langle v_1,v_2\rangle_V,\quad u_1,u_2\in U,\ v_1,v_2\in V.
\end{equation*}
If $(u_n)_{n\in\N}\subset U$ and $(v_n)_{n\in\N}\subset V$ are ON-bases, then $(u_m\otimes v_n)_{(m,n)\in\N^2}\subset U\otimes V$ is an ON-basis. The space $U\otimes V$ can be realized in several isomorphic ways. If the tensor product $u\otimes v$ realizes a rank one operator $(u\otimes v)\phi=\langle v,\phi\rangle_Vu$ for $\phi\in V$, then $U\otimes V\cong\LB_2(V,U)$. If $U$ and $V$ are spaces of real-valued functions of independent variables $x\in\D_1$ and $y\in\D_2$ respectively, then $(u\otimes v)(x,y)=u(x)v(y)$ is also a realization of $U\otimes V$. For instance, if $U=L_2(\D)$ and $V=L_2(\Omega)$, where $\D$ is our spatial domain and $\Omega$ the sample space, then $U\otimes V=L_2(\Omega\times\D)\cong L_2(\Omega,L_2(\D))$, i.e., $L_2(\D)$-valued square integrable random variables. For a detailed introduction to tensor products, see \cite[App.~E]{Janson}.

\subsection{Fr\'{e}chet derivatives}\label{subsec:frechet}
\label{Frechet}
Let $(U,\|\cdot\|_U)$ and $(V,\|\cdot\|_V)$ be Banach spaces. By $\Cb^m(U,V)$ we denote the space of not necessarily bounded mappings $g\colon U\rightarrow V$ having $m\ge1$ continuous and bounded Fr\'{e}chet derivatives $Dg, D^2g,\dots,D^m g$. We endow it with the seminorm $|\cdot|_{\Cb^m(U,V)}$, defined as the smallest constant $C\geq0$ such that
\begin{equation*}
\sup_{x\in U}\|D^mg(x)\cdot(\phi_1,\dots,\phi_m)\|_{V}\leq C\|\phi_1\|_U\cdots \|\phi_m\|_U,\quad \forall\phi_1,\dots,\phi_m\in U.
\end{equation*}
It will be convenient to write $\Cb^m=\Cb^m(U,V)$. From the context it will be clear what we mean.

Let us consider the important case when $U$ is a Hilbert space and $V=\R$. The Fr\'{e}chet derivative $Dg(x)$ of a function $g\colon U\rightarrow \R$ is a bounded linear functional on $U$ for fixed $x\in H$ and it can thus be identified by its gradient using the Riesz representation theorem, i.e., $Dg(x)\cdot \phi=\langle Dg(x),\phi\rangle$. In the same way the second derivative enjoys a representation as a bounded linear operator by the identity $D^2g(x)\cdot(\phi,\psi)=\langle D^2g(x)\phi,\psi\rangle$. We will use both representations and it will lead to no confusion.

\subsection{The functional analytic framework} \label{subsec:functional}
We introduce the semigroup framework on which our analysis of equations \eqref{mild} and \eqref{SPDEapprox} relies. Recall from Section 1 that $A=-\Delta$ with $\dom(A)=H^2(\D)\cap H_0^1(\D)$ and $H=L_2(\D)$ with $\D\subset\R^d$ a convex polygonal domain. We denote $\|\cdot\|=\|\cdot\|_H$, $\langle\cdot,\cdot\rangle=\langle\cdot,\cdot\rangle_H$, $\LB=\LB(H)$ and $\LB_p=\LB_p(H)$. The operator $A$ is closed, selfadjoint, and positive definite with compact inverse.

There is an orthonormal eigenbasis $(\varphi_i)_{i\in\N}\subset H$ with corresponding eigenvalues $0<\lambda_1<\lambda_2\leq\dots\leq\lambda_i\rightarrow\infty$, as $i\rightarrow\infty$, for which $A\varphi_i=\lambda_i\varphi_i$, $i\in\N$. The asymptotics $\lambda_i\sim i^{2/d}$, as $i\rightarrow\infty$, is well known. When the space dimension $d=1$, as in Assumption \textbf{B}, we have
\begin{equation}
\label{Awhite}
\Tr(A^{-\frac12\gamma})=\|A^{-\frac12\gamma}\|_{\LB_1} =\|A^{-\frac14\gamma}\|_{\LB_2}^2<\infty,\quad\forall\gamma>1,\ \textrm{if}\ d=1.
\end{equation}
Therefore, $\beta\in(0,\tfrac12)$ in Assumption \textbf{B}.

We define norms of fractional orders by
\begin{equation*}
\|v\|_{\dot{H}^\beta}=\|A^{\frac\beta2}v\| =\Big(\sum_{i\in\N}\lambda_i^\beta\langle v,\varphi_i\rangle^2\Big)^{\frac12},\quad\beta\in\R.
\end{equation*}
The spaces $\dot{H}^\beta$ are then, for $\beta\geq0$, defined as $\dom(A^{\frac\beta2})$ and for $\beta<0$ as the closure of $H$ with respect to the $\dot{H}^\beta$-norm. The space $\dot{H}^{-\gamma}$ of negative order can be identified with the dual space of $\dot{H}^\gamma$. Clearly $\dot{H}^0=H$, and it is also well known that $\dot{H}^1=H_0^1(\D)$ and $\dot{H}^2=H^2(\D)\cap H_0^1(\D)$, see \cite[Ch. 3]{Thomee}.

Let $(V_h)_{h\in(0,1)}$ denote a family of standard finite element spaces of continuous piecewise linear functions corresponding to a quasi-uniform family of triangulations, for which $h$ denotes the largest diameter in the triangulation. Then $V_h\subset\dot{H}^1$. By $P_h$ we denote the orthogonal projector of $H$ onto $V_h$. Let $A_h\colon V_h\rightarrow V_h$ be the unique operator satisfying
\begin{equation*}
\langle A_h\psi,\chi\rangle=\langle\nabla\psi,\nabla\chi\rangle,\quad\forall\psi,\chi\in V_h.
\end{equation*}
This is the discrete Laplacian. By definition
\begin{equation}
\label{NormRel1}
\|A_h^{\frac12}\varphi_h\|=\|\nabla \varphi_h\|
=\|A^{\frac12}\varphi_h\|=\|\varphi_h\|_{\dot{H}^1},\quad \varphi_h\in V_h.
\end{equation}
Therefore, $P_h$ can be extended to $\dot{H}^{-1}$, so that for all $\varphi\in\dot{H}^{-1}$,
\begin{equation}
\label{NormRel2}
\|A_h^{-\frac12}P_h\varphi\|=\sup_{\psi\in V_h}\frac{\langle \varphi,\psi\rangle}{\|A_h^\frac12\psi\|}
=\sup_{\psi\in V_h}\frac{\langle \varphi,\psi\rangle}{\|A^\frac12\psi\|}
\leq\sup_{\psi\in \dot{H}^1}\frac{\langle \varphi,\psi\rangle}{\|A^\frac12\psi\|}=\|A^{-\frac12}\varphi\|.
\end{equation}
In the following, $c$ and $C$ denote various constants that do not depend on $h$. From \eqref{NormRel1} and the well-known fact that $P_h$ is bounded with respect to $\|\cdot\|_{\dot{H}^1}=\|A^{\frac12}\cdot\|$ when we use a quasi-uniform mesh family, we obtain
\begin{equation}
\label{Ph}
\|A_h^{\frac12}P_h\varphi\|\leq C\|A^{\frac12}\varphi\|,\quad \varphi\in\dot{H}^1.
\end{equation}
Interpolation between this and \eqref{NormRel2} yields
\begin{equation}
\label{leq3}
\|A_h^\gamma P_h\varphi\|\leq C\|A^\gamma\varphi\|,\quad\varphi\in\dot{H}^\gamma,\ \gamma\in[-\tfrac12,\tfrac12].
\end{equation}
Furthermore, \eqref{Ph} means that $\|A_h^{\frac12}P_hA^{-\frac12}\|_{\LB}\le C$. Hence,
\begin{equation*}
\|A^{-\frac12}A_h^{\frac12}P_h\|_{\LB} =\|(A^{-\frac12}A_h^{\frac12}P_h)^*\|_{\LB} =\|A_h^{\frac12}P_hA^{-\frac12}\|_{\LB}\leq C,
\end{equation*}
so that $\|A^{-\frac12}A_h^{\frac12}P_h\varphi\|\leq C\|\varphi\|$ or
\begin{equation*}
\|A^{-\frac12}\varphi_h\|\leq C\|A_h^{-\frac12}\varphi_h\|,\quad \varphi_h\in V_h.
\end{equation*}
Interpolating between this and \eqref{NormRel1} yields
\begin{equation*}
\label{AhAhigh}
\|A^\gamma\varphi_h\|\leq C\|A_h^{\gamma}\varphi_h\|,\quad \varphi_h\in V_h,\ \gamma\in[-\tfrac12,\tfrac12].
\end{equation*}
Using also \eqref{leq3} yields the norm equivalence
\begin{equation}
\label{eqnorm}
c\|A_h^{\gamma}\varphi_h\|\leq\|A^{\gamma}\varphi_h\|\leq C\|A_h^{\gamma}\varphi_h\|,\quad \varphi_h\in V_h,\  \gamma\in[-\tfrac12,\tfrac12].
\end{equation}
The interpolations above are valid since $(\dot{H}^\beta)_{\beta\in[-1,1]}$ and $(\dot{H}_h^\beta)_{\beta\in[-1,1]}$ are real interpolation spaces, where $\dot{H}_h^\beta=V_h$ with norm $\|v_h\|_{\dot{H}_h^\beta}=\|A_h^{\frac\beta2}v_h\|$. For positive order this is standard, see for instance \cite{Lunardi}. For negative order, let $\beta\in[0,1]$ and notice that
\begin{equation*}
[\dot{H}^0,\dot{H}^{-1}]_{\beta,2}=[(\dot{H}^0)^*,(\dot{H}^1)^*]_{\beta,2} =[\dot{H}^0,\dot{H}^1]_{\beta,2}^* =(\dot{H}^\beta)^*=\dot{H}^{-\beta}.
\end{equation*}

We define the Ritz projector $R_h\colon \dot{H}^1\rightarrow V_h$ to be the orthogonal projection with respect to the $\dot{H}^1$-scalar product. Since $\D$ is convex and polygonal it is well known that
\begin{equation}
\label{Rhleq}
\|A^{\frac{s}2}(I-R_h)A^{-\frac{r}2}\|_{\LB}\leq Ch^{r-s},\quad 0\leq s\leq1\leq r\leq2.
\end{equation}
For $P_h$ the following error estimate holds
\begin{equation}
\label{Phleq}
\|A^{\frac{s}2}(I-P_h)A^{-\frac{r}2}\|_{\LB}\leq Ch^{r-s},\quad 0\leq s\leq1,\quad 0\leq s\leq r\leq2.
\end{equation}
For more about the finite element method, see \cite{Brenner} for elliptic and \cite{Thomee} for parabolic equations.

Denote by $N_h$ the dimension of $V_h$. There is an orthonormal eigenbasis $(\varphi_i^h)_{i=1}^{N_h}\subset V_h$ corresponding to $A_h$ with eigenvalues $0<\lambda_1^h\leq\lambda_2^h\leq\dots\leq\lambda_{N_h}^h$. The operators $-A$ and $-A_h$ generate analytic semigroups $(S(t))_{t\geq0}$ and $(S_h(t))_{t\geq0}$, respectively. They are spectrally given by
\begin{equation}
\label{semigroup}
S(t)v=\sum_{i\in\N}e^{-\lambda_it}\langle v,\varphi_i\rangle\varphi_i,\quad v\in H,\ t\geq0,
\end{equation}
and
\begin{equation*}
S_h(t)v_h=\sum_{i=1}^{N_h}e^{-\lambda_i^ht}\langle v_h,\varphi_i^h\rangle\varphi_i^h,\quad v_h\in V_h,\ t\geq0.
\end{equation*}
The semigroup $(S_h(t))_{t\geq0}$ solves the parabolic equation $\dot{u}_h+A_hu_h=0$, $t\geq0$, with $u_h(0)=v_h$, in the sense that $u_h(t)=S_h(t)v_h$.

Important for our analysis is the estimate, see \cite{Thomee},
\begin{equation}
\label{AEhleq}
\|A^{\gamma}S(t)\|_{\LB}+\|A_h^{\gamma}S_h(t)P_h\|_{\LB}\leq C_\gamma t^{-\gamma}, \quad \gamma\geq0,\ t>0,\quad\text{uniformly in $h$}.
\end{equation}
This inequality is characteristic for analytic semigroups.

Let $P_m$ denote the spectral projection onto the space spanned by the $m$ first eigenvectors $(\varphi_i)_{i=1}^m$ of $A$. An easy calculation shows that
\begin{align}
\label{Pmleq}
\|(I-P_m)A^{-r}\|_{\LB}&\leq \lambda_{m+1}^{-r},\quad r\geq0,\\
\label{Pmleq2}
\|P_mA^r\|_\LB&\leq \lambda_{m+1}^r,\quad r\geq0.
\end{align}

We will frequently use the following generalized Gronwall lemma:
\begin{lemma}
\label{Gronwall}
Let $\vartheta(t)\geq0$ be a continuous function on $[0,T]$. If, for some $A,B\geq0$ and $\alpha,\beta\in[0,1)$, the inequality
\begin{equation*}
\vartheta(t)\leq At^{-\alpha}+B\int_0^t(t-s)^{-\beta}\vartheta(s)\diff{s}, \quad t\in[0,T],
\end{equation*}
holds, then there is $C=C(B,T,\alpha,\beta)$ such that
\begin{equation*}
\vartheta(t)\leq CAt^{-\alpha},\quad t\in(0,T].
\end{equation*}
\end{lemma}

In our analysis we will use the notation $a\lesssim b$, to mean that there exists a constant $C>0$ such that $a\leq Cb$. The constant will never depend on the mesh size $h$.

\subsection{The stochastic integral and Malliavin calculus} \label{subsec:malliavin}
Since we use the Malliavin calculus in the proof of our main result, we outline a framework for the stochastic integral in which this calculus has a natural role. This is an alternative to the more classical procedure, presented in \cite{DaP}. Our presentation of the Wiener integral relies on \cite{EEB}, and the Malliavin calculus on \cite{MalliavinRef1} and \cite{MalliavinRef2}, where a natural extension of the framework of \cite{Nualart} to Hilbert space valued stochastic integrals using tensor products is presented.

The covariance operator $Q\in\LB(H)$ is selfadjoint and positive semidefinite. Let $Q^{1/2}$ denote the unique positive square root. Let $Q^{-1/2}$ be its inverse, restricted to $(\ker Q)^\perp$. Define the Hilbert space $U_0=Q^{1/2}(H)$, equipped with the scalar product $\langle u,v\rangle_{U_0}=\langle Q^{-1/2}u,Q^{-1/2}v\rangle$.  If $\Tr(Q)<\infty$, then the triple $i\colon U_0\hookrightarrow H$ is an abstract Wiener space, where $i$ is the inclusion mapping $i\colon x\mapsto x$. This triple induces a Gaussian probability measure on $H$ with mean $0$ and covariance $Q$. It is referred to as an abstract Wiener measure. The space $U_0$ is called the Cameron-Martin space in this context.

Let $I\colon L_2([0,T],U_0)\rightarrow L_2(\Omega)$ be an isonormal process, i.e., for every $\phi\in L_2([0,T],U_0)$ the random variable $I(\phi)$ is centered Gaussian and $I$ has the covariance structure
\begin{equation*}
\E[I(\phi)I(\psi)]=\langle \phi,\psi\rangle_{L_2([0,T],U_0)}, \quad\phi,\psi\in L_2([0,T],U_0).
\end{equation*}
The existence of $I$ follows by an application of the Kolmogorov Extension Theorem.

Define, for $u\in U_0$, the cylindrical $Q$-Wiener process $W\in L_2( [0,T]\times U_0, L_2(\Omega))$ by
\begin{equation*}
W(t)u=I(\chi_{[0,t]}\otimes u),\quad u\in U_0,\ t\in[0,T].
\end{equation*}
For $u\in U_0$ the process $W(t)u$, $t\in[0,T]$, is a Brownian motion and
\begin{equation*}
\E[W(t)uW(s)v]=\min(s,t)\langle u,v\rangle_{U_0}, \quad u,v\in U_0.
\end{equation*}

The space of Hilbert-Schmidt operators $\LB_2^0=\LB_2(U_0,H)$ can be identified with $H\otimes U_0$ with $h\otimes u\in \LB_2^0$ for $h\in H$, $u\in U_0$, being the operator $(h\otimes u)v=\langle u,v\rangle_{U_0}h$, $v\in U_0$.

We now define the $H$-valued Wiener integral for the simplest possible integrands. Let $\Phi=\chi_{[a,b]}\otimes(h\otimes u)\in L_2([0,T],\LB_2^0)$, for $a,b\in[0,T]$, $h\in H$ and $u\in U_0$. Then the Wiener integral of $\Phi$ is defined as the $H$-valued random variable
\begin{equation*}
\int_0^T\Phi(s)\diff{W(s)}=I(\chi_{[a,b]}\otimes u)\otimes h=\big(W(b)u-W(a)u\big)\otimes h\in L_2(\Omega,H).
\end{equation*}
It is easy to show that for such integrands the following Wiener isometry holds:
\begin{equation*}
\E\Big[\Big\|\int_0^T\Phi(t)\diff{W(t)}\Big\|_H^2\Big] =\int_0^T\|\Phi(t)\|_{\LB_2^0}^2\diff{t}.
\end{equation*}
The integral extends directly to linear combinations of such integrands by linearity of $I$. By the Wiener isometry, the completeness of $L_2([0,T],\LB_2^0)$, and classical approximation results for $L_2([0,T])$-functions and for compact operators, it extends to all $\Phi\in L_2([0,T],\LB_2^0)$.

Let $\mathcal{C}_{\mathrm{p}}^\infty(\R^n)$ denote the space of all real-valued $\mathcal{C}^\infty$-functions on $\R^n$ with polynomial growth. We define the family of smooth cylindrical random variables
\begin{equation*}
\mathcal{S}=\big\{X=f(I(\phi_1),\dots,I(\phi_N)): f\in \mathcal{C}_{\mathrm{p}}^\infty(\R^N),\;\phi_1,\dots,\phi_N\in L_2([0,T],U_0),\;N\geq1\big\}
\end{equation*}
and the corresponding family with values in $H$ as
\begin{equation*}
\mathcal{S}(H)=\Big\{F=\sum_{i=1}^MX_i\otimes h_i: X_1,\dots,X_M\in\mathcal{S},\;h_1,\dots,h_M\in H,\;M\geq1\Big\}.
\end{equation*}
The Malliavin derivative of a random variable in $X=f(I(\phi_1),\dots,I(\phi_N))\in\mathcal{S}$ is defined as the $L_2([0,T],U_0)$-valued random variable $DX=\sum_{i=1}^N\partial_if(I(\phi_1),\dots,I(\phi_N))\otimes\phi_i$. Clearly, this is a $U_0$-valued stochastic process. We write $D_tX=\sum_{i=1}^N\partial_if(I(\phi_1),\dots,I(\phi_N))\otimes\phi_i(t)$ for $t\in[0,T]$. The Malliavin derivative of a random variable $F=\sum_{i=1}^Mf_i(I(\phi_1),\dots,I(\phi_N))\otimes h_i\in\mathcal{S}(H)$ is given by
\begin{equation*}
D_tF=\sum_{i=1}^M\sum_{j=1}^N\partial_jf_i(I(\phi_1),\dots,I(\phi_N)) \otimes(h_i\otimes\phi_j(t)).
\end{equation*}
Thus $(D_tF)_{t\in[0,T]}$ is an $\LB_2^0$-valued stochastic process. By $D_t^uF$ we denote the derivative of $F$ in the direction $u\in U_0$ at time $t$, i.e., $D_t^uF=D_tFu$, where
\begin{equation*}
D_tFu=\sum_{i=1}^M\sum_{j=1}^N\langle u,\phi_j(t)\rangle_{U_0}\,\partial_jf_i(I(\phi_1),\dots,I(\phi_N))\otimes h_i.
\end{equation*}
At the very heart of Malliavin calculus is the following integration by parts formula:
\begin{equation}
\label{smoothIBP}
\E\langle DF,\Phi\rangle_{L_2([0,T],\LB_2^0)}=\E\Big\langle F,\int_0^T\Phi(t)\diff{W(t)}\Big\rangle_H,\quad F\in\mathcal{S}(H),\quad\Phi\in L_2([0,T],\LB_2^0).
\end{equation}
Thus, the Wiener integral is the adjoint of $D\colon\mathcal{S}(H)\subset L_2(\Omega,H)\rightarrow L_2(\Omega\times[0,T],\LB_2^0)$ for deterministic integrands. Formula \eqref{smoothIBP} follows from the corresponding formula for real-valued smooth stochastic variables. The derivative operator $D$ is known to be closable. We define the Watanabe Sobolev space $\mathbf{D}^{1,2}(H)$ as the closure of $\mathcal{S}(H)$ with respect to the norm
\begin{equation*}
\|F\|_{\mathbf{D}^{1,2}(H)}=\Big(\E\big[\|F\|_H^2\big]+\E\Big[\int_0^T \|D_tF\|_{\LB_2^0}^2\diff{t}\Big]\Big)^{\frac12}.
\end{equation*}

Denote by $\mathrm{dom}(\delta)$ the elements $\Phi\in L_2(\Omega\times[0,T],\LB_2^0)$ for which $\E[\langle DF,\Phi\rangle_{L_2([0,T],\LB_2^0)}]$ defines a bounded linear functional acting on $F\in\mathbf{D}^{1,2}(H)$. For any such $\Phi$ the functional $l_\Phi(F)=\E[\langle DF,\Phi\rangle_{L_2([0,T],\LB_2^0)}]$ can be extended by continuity to all $F\in L_2(\Omega,H)$. The Riesz representation theorem guarantees the existence of an adjoint operator to $D$, namely $\delta\colon \mathrm{dom}(\delta)\subset L_2(\Omega\times[0,T],\LB_2^0)\rightarrow L_2(\Omega,H)$ that satisfies
\begin{equation}
\label{deltaIBP}
\E[\langle DF,\Phi\rangle_{L_2([0,T],\LB_2^0)}]=\E[\langle F,\delta(\Phi)\rangle_H],\quad\forall F\in\mathbf{D}^{1,2}(H),\ \Phi\in\dom(\delta).
\end{equation}
This is a natural extension of \eqref{smoothIBP} to a much larger class of integrands. In \cite[Lemme 2.10]{MalliavinRef1} it is proved that for any predictable process $\Phi\in L_2(\Omega\times[0,T],\LB_2^0)$ the action of $\delta$ on $\Phi$ coincides with that of the It\^{o} integral, i.e.,
\begin{equation*}
\delta(\Phi)=\int_0^T\Phi(t)\diff{W(t)}.
\end{equation*}
Instead of relying on It\^{o} theory we take this as the definition of the It\^{o} integral. We remark that $\textrm{dom}(\delta)$ contains processes that are not predictable and thus $\delta$ is an extension of the It\^{o} integral to such integrands. In this context $\delta$ is called the Skorohod integral.

The following lemma \cite[Lemma 2.1]{Debussche} has a central role in the proof of our main result.
\begin{lemma}
\label{IBP}
For any random variable $F\in\mathbf{D}^{1,2}(H)$ and any predictable process $\Phi\in L_2([0,T]\times\Omega,\LB_2^0)$ the following integration by parts formula is valid:
\begin{equation*}
\E\Big[\Big\langle \int_0^T\Phi(t)\diff{W(t)},F\Big\rangle_H\Big] =\E\Big[\int_0^T\langle\Phi(t),D_tF\rangle_{\LB_2^0}\diff{t}\Big].
\end{equation*}
\end{lemma}

\begin{proof}
This is just a restatement of \eqref{deltaIBP} for predictable $\Phi$.
\end{proof}

A corollary of Lemma \ref{IBP} is the It\^{o} isometry. It reads
\begin{equation}
\label{ItoIsometry}
\E\Big[\Big\|\int_0^T\Phi(t)\diff{W(t)}\Big\|_H^2\Big] =\E\Big[\int_0^T\|\Phi(t)\|_{\LB_2^0}^2\diff{t}\Big], \quad\forall\Phi\in L_2([0,T]\times\Omega,\LB_2^0),\ \textrm{predictable}.
\end{equation}
The Malliavin derivative acts on its adjoint by $D_s^u\delta(\Phi)=\delta(D_s^u\Phi)+\Phi(s)u$, or in terms of the It\^{o} integral $\delta(\chi_{[0,t]}\Phi)=\int_0^t\Phi(r)\diff{W(r)}$ with predictable $\Phi\in L_2([0,T]\times\Omega,\LB_2^0)$ satisfying $\Phi(t)\in\mathbf{D}^{1,2}(\LB_2^0)$ for all $t\in[0,T]$:
\begin{equation}
\label{DIto}
D_s^u\int_0^t\Phi(r)\diff{W(r)}=\int_0^tD_s^u\Phi(r)\diff{W(r)}+\Phi(s)u,\quad 0\leq s\leq t\leq T.
\end{equation}
If $s>t$, then $D_s^u\int_0^t\Phi(r)\diff{W(r)}=0$, since the integral is $\mathcal{F}_t$-measurable. The class of $F\in\mathbf{D}^{1,2}(H)$ that are $\mathcal{F}_0$-measurable coincides with the class of constant (deterministic) random variables.  Let $V$ be another separable real Hilbert space and $\sigma\in\Cb^1(H,V)$. Then $\sigma(F)\in\mathbf{D}^{1,2}(V)$ and we have the chain rule
\begin{align}
\label{Chain1}
D_t^u(\sigma(F))&=D\sigma(F)\cdot D_t^uF,\quad u\in U_0,\ F\in\mathbf{D}^{1,2}(H),\\
\label{Chain2}
D_t(\sigma(F))&=D\sigma(F)D_tF,\qquad F\in\mathbf{D}^{1,2}(H).
\end{align}

\subsection{Existence and uniqueness} \label{subsec:existence} Existence and uniqueness of a solution to \eqref{mild}, under Assumption \textbf{A} with $\beta=1$, is stated in \cite[Theorem 7.4]{DaP}. This is the case when $\Tr(Q)<\infty$. The extension to $\beta\in[\tfrac12,1)$ is straight-forward. Existence and uniqueness under case \textbf{B} is given in \cite[Theorem 7.6]{DaP}. By using the methods of \cite{MR2852200} and \cite{KruseLarsson} one can show that the regularity in space is of order $\beta$, i.e., the solution $X$ is of the form $[0,T]\times\Omega\rightarrow\dot{H}^\beta$, $\mathbf{P}$-a.s. Recall here that $\beta\in[\tfrac12,1]$ under Assumption \textbf{A} and $\beta\in(0,\tfrac12)$ under Assumption \textbf{B}. The existence of the family $(X_h)_{h\in(0,1)}$ of solution processes of the discrete equation \eqref{SPDEapprox}, is proved analogously and clearly $X_h(t)\in V_h\subset\dot{H}^1$, $\mathbf{P}$-a.s. Recalling \eqref{eqnorm}, the estimate $\E\|A^{\frac\gamma2}X_h(t)\|^2\leq C(1+\|A^\frac\gamma2X_0\|^2)$, uniformly in $h$, holds only for $\gamma\in[0,\beta]$. The strong convergence
\begin{equation} \label{eq:kruse-strong}
\big(\E\|X(T)-X_h(T)\|^2\big)^{\frac12}\leq Ch^{\beta},
\end{equation}
is proved in \cite{Kruse} under the assumption of trace class noise. The proof is similar under Assumptions \textbf{A} and \textbf{B}. We formulate a qualitative bound for the solution processes in the following theorem. We remind the reader that $X_0$ is deterministic.
\begin{theorem}
\label{Existence}
Under either Assumption \textbf{A} or Assumption \textbf{B} there exist unique predictable solutions $X\in \mathcal{C}([0,T],L_2(\Omega,H))$ and $X_h\in \mathcal{C}([0,T],L_2(\Omega,V_h))$ to equation \eqref{mild} and \eqref{SPDEapprox} respectively. We refer to these solutions as the unique mild solutions of \eqref{SPDE} and \eqref{SPDEapprox}. There exists a constant $C$, such that the following moment estimates hold
\begin{equation}
\label{moment}
\sup_{t\in[0,T]}\E\|X(t)\|^2+\sup_{t\in[0,T]}\E\|X_h(t)\|^2\leq C(1+\|X_0\|^2).
\end{equation}
\end{theorem}

\section{Estimates of the Malliavin derivative of the solution} \label{sec:estimates}
We consider the Malliavin derivative of the discrete solution process and prove some estimates needed later. Differentiating the equation \eqref{SPDEapprox} formally in direction $u\in U_0$, using \eqref{DIto}, \eqref{Chain1}, and the fact that we have a deterministic initial value, yields
\begin{equation}
\label{DsXh}
\begin{split}
D_s^{u}X_h(t)&=S_h(t-s)P_h\g(X_h(s))u+\int_s^tS_h(t-r)P_hDf(X_h(r))\cdot D_s^uX_h(r)\diff{r}\\
&\quad+\int_s^tS_h(t-r)P_h\big(Dg(X_h(r))\cdot D_s^uX_h(r)\big)\diff{W(r)},\quad 0\leq s\leq t\leq T.
\end{split}
\end{equation}
This equation is treated much like \eqref{SPDEapprox} itself. It has a unique solution.

Before we proceed to the estimate of the Malliavin derivative, we notice that, by the linear growth of $f$ and $\g$, implied by their bounded first derivatives, and the moment estimate \eqref{moment} for $X$ and $X_h$, we obtain
\begin{equation}
\label{Bndfg}
\begin{split}
\sup_{t\in[0,T]}\E\|f(Y(t))\|^2 +\sup_{t\in[0,T]}\E\|g(Y(t))\|_{\LB}^2\lesssim 1+\|X_0\|^2,\quad Y=X\;\mathrm{ or }\;X_h.
\end{split}
\end{equation}

\begin{lemma}
\label{LemmaDsXh}
Consider equation \eqref{SPDEapprox} under Assumption \textbf{A}. Then the Malliavin derivative of $X_h$, given as the solution $D_sX_h$ to equation \eqref{DsXh}, satisfies for some constant $C=C(T)>0$ the bound:
\begin{equation*}
\E\big[\|A_h^{\frac{\beta-1}2}D_sX_h(t)\|_{\LB_2^0}^2\big]\leq C,\quad 0\leq s\leq t\leq T.
\end{equation*}
\end{lemma}

\begin{proof}
We use \eqref{DsXh} with $g(x)=I$, $Dg(x)=0$ and recall from Assumption \textbf{A} that $\beta-1\in[-\tfrac12,0]$. Fix $u\in U_0$. Then
\begin{equation*}
\begin{split}
\E\|D_s^uX_h(t)\|^2&\lesssim\|S_h(t-s)A_h^{\frac{1-\beta}2} A_h^{\frac{\beta-1}2}P_hu\|^2+\int_s^t\E\|S_h(t-r)P_hDf(X_h(r))\cdot D_s^uX_h(r)\|^2\diff{r}.
\end{split}
\end{equation*}
In view of \eqref{leq3} and the boundedness of $Df$ and $S_h(t)$ we have
\begin{equation}
\label{MDest2}
\E\|D_s^uX_h(t)\|^2\lesssim \|A_h^{\frac{1-\beta}2}S_h(t-s)P_h\|_{\LB}^2\,\|A^{\frac{\beta-1}2}u\|^2 +\int_s^t|f|_{\Cb^1}^2\,\E\|D_s^uX_h(r)\|^2\diff{r}.
\end{equation}
The analyticity of the semigroup \eqref{AEhleq} yields
\begin{equation*}
\E\|D_s^uX_h(t)\|^2\lesssim(t-s)^{\beta-1}\|A^{\frac{\beta-1}2}u\|^2 +\int_s^t\E\|D_s^uX_h(r)\|^2\diff{r}
\end{equation*}
and applying Gronwall's Lemma \ref{Gronwall}, for fixed $s\in[0,t)$, gives
\begin{equation}
\label{DsXhleq}
\E\|D_s^uX_h(t)\|^2\lesssim(t-s)^{\beta-1}\|A^{\frac{\beta-1}2}u\|^2.
\end{equation}
Proceeding as in the proof of \eqref{MDest2}, we obtain also
\begin{equation*}
\E\|A_h^{\frac{\beta-1}2}D_s^uX_h(t)\|^2\lesssim \|A^{\frac{\beta-1}2}u\|^2+\int_s^t\E\|D_s^uX_h(r)\|^2\diff{r}.
\end{equation*}
Estimate \eqref{DsXhleq} is applicable here. Thus,
\begin{equation*}
\begin{split}
\int_s^t\E\|D_s^uX_h(r)\|^2\diff{r}&\lesssim \int_s^t(r-s)^{\beta-1}\diff{r} \;\|A^{\frac{\beta-1}2}u\|^2\lesssim(t-s)^{\beta}\|A^{\frac{\beta-1}2}u\|^2
\end{split}
\end{equation*}
and hence
\begin{equation}
\label{MDest5}
\E\|A_h^{\frac{\beta-1}2}D_s^uX_h(t)\|^2\lesssim \|A^{\frac{\beta-1}2}u\|^2.
\end{equation}
Notice that this is uniform with respect to $u\in U_0$. We take an ON-basis $(u_i)_{i\in\N}\subset U_0$ and compute the $\LB_2^0$-norm according to \eqref{defHS2}. Using Tonelli's Theorem and \eqref{MDest5} we get that
\begin{equation*}
\begin{split}
\E\|A_h^{\frac{\beta-1}2}D_sX_h(t)\|_{\LB_2^0}^2 &=\E\sum_{i\in\N}\|A_h^{\frac{\beta-1}2}D_s^{u_i}X_h(t)\|^2 =\sum_{i\in\N}\E\|A_h^{\frac{\beta-1}2}D_s^{u_i}X_h(t)\|^2\\
&\lesssim \sum_{i\in\N}\|A^{\frac{\beta-1}2}u_i\|^2 =\|A^{\frac{\beta-1}2}\|_{\LB_2^0}^2.
\end{split}
\end{equation*}
This completes the proof.
\end{proof}

For the white noise case we need the following lemma which is a spatially discrete analogue of \cite[Lemma 4.3]{Debussche}. Recall that in this case $Q=I$, $U_0=H$, $\LB_2^0=\LB_2$.
\begin{lemma}
\label{Dsleq}
Consider equation \eqref{SPDEapprox} under Assumption \textbf{B}. Then, for $\gamma\in[0,\tfrac12)$, the Malliavin derivative satisfies the following estimate:
\begin{equation*}
\E\|A_h^{\frac{\gamma}2}D_sX_h(t)\|_{\LB}^2\leq C(t-s)^{-\gamma},\quad 0\leq s< t\leq T.
\end{equation*}
\end{lemma}
\begin{proof}
Let $u\in H$, and take norms in \eqref{DsXh} using the Cauchy-Schwarz inequality and the It\^{o} isometry \eqref{ItoIsometry} to get
\begin{equation*}
\begin{split}
\E\|A_h^{\frac{\gamma}2}D_s^uX_h(t)\|^2&\lesssim \E\|A_h^{\frac{\gamma}2}S_h(t-s)P_h\g(X_h(s))u\|^2\\
&\quad+\int_s^t\E\|A_h^{\frac\gamma2}S_h(t-s)P_hDf(X_h(s))\cdot D_s^uX_h(r)\|^2\diff{r}\\
&\quad+\int_s^t\E\|A_h^{\frac14+\epsilon}A_h^{\frac\gamma2} S_h(t-s)A_h^{-\frac14-\epsilon}P_hDg(X_h(s))\cdot D_s^uX_h(r)\|_{\LB_2}^2\diff{r}.
\end{split}
\end{equation*}
For $\epsilon>0$ small enough we have by \eqref{HolderInf} and \eqref{AEhleq}
\begin{equation*}
\begin{split}
\E\|A_h^{\frac{\gamma}2}D_s^uX_h(t)\|^2&\lesssim (t-s)^{-\gamma}\sup_{s\in[0,T]}\E\|\g(X_h(s))\|_{\LB}^2\|u\|^2\\
&\quad+\int_s^t(t-r)^{-\gamma}\;|f|_{\Cb^1}^2\; \E\|A_h^{\frac\gamma2}D_s^uX_h(r)\|^2\diff{r}\\
&\quad+\int_s^t(t-r)^{-\gamma-\frac12-2\epsilon} \|A_h^{-\frac14-\epsilon}P_h\|_{\LB_2}^2\;|g|_{\Cb^1}^2\; \E\|A_h^{\frac\gamma2}D_s^uX_h(r)\|^2\diff{r}.
\end{split}
\end{equation*}
Here by \eqref{Awhite}, \eqref{HolderInf} and \eqref{leq3} we have
\begin{equation*}
\|A_h^{-\frac14-\epsilon}P_h\|_{\LB_2}\lesssim\|A_h^{-\frac14-\epsilon}P_h A^{\frac14+\epsilon}\|_{\LB}\|A^{-\frac14-\epsilon}\|_{\LB_2} \lesssim\|A^{-\frac14-\epsilon}\|_{\LB_2}<\infty.
\end{equation*}
Finally, by Lemma \ref{Gronwall} and \eqref{Bndfg}, we conclude
\begin{equation*}
\E\|A_h^{\frac{\gamma}2}D_s^uX_h(t)\|^2\lesssim (t-s)^{-\gamma}(1+\|X_0\|^2)\|u\|^2.
\end{equation*}
This completes the proof.
\end{proof}

\section{Regularity results for the Kolmogorov equation} \label{sec:kolmogorov}

In \cite{Brehier}, \cite{Debussche}, and \cite{WangGan}, weak convergence estimates are proved for pure time discretization. The use of the It\^{o} formula and the Kolmogorov equation in the proofs is justified by making a finite-dimensional spectral Galerkin approximation. The estimates are uniform with respect to the dimension of the approximation space and therefore holds in the limit. This approximation is not made explicit in the proofs there. For spatial discretization we need to take more care. This is because the operators $P_m$ and $A_h$ do not commute.

Recall that $P_m$ is the projection onto the subspace of $H_m\subset H$ spanned by the first $m\in\N$ eigenvectors $(\varphi_i)_{i=1}^m$ of $A$. Let $A_m=P_mAP_m=AP_m=P_mA$. By $(S_m(t))_{t\geq0}$ we denote the semigroup generated by $-A_m$, i.e., it is given by the $m$ first terms in the spectral representation \eqref{semigroup} of $(S(t))_{t\geq0}$.

We denote by $X_m^x$ the solution of
\begin{equation*}
X_m^x(t)=S_m(t)P_m x+\int_0^tS_m(t-s)P_m f(X_m^x(s))\diff{s}+\int_0^tS_m(t-s)P_mg(X_m^x(s))\diff{W(s)}, \ t\in[0,T].
\end{equation*}
Let $\varphi\in\Cb^2(H,\R)$ and define the function $u_m(t,x)=\E[\varphi(X_m^x(t))]$ for $t\in[0,T]$, $x\in H$. Note that $u(t,x)=u(t,P_mx)$ for $x\in H$. It is well known, see, e.g., \cite[Theorem 9.16]{DaP}, that $u_m\colon[0,T]\times H\rightarrow \R$ is a solution to the Kolmogorov equation
\begin{equation}\label{kolmo}
\begin{array}{ll}
\dot{u}_m(t,x)+L_mu_m(t,x)=0, \quad &(t,x)\in(0,T]\times H,\\
u_m(0,x)=\varphi(P_mx), \quad &x\in H,
\end{array}
\end{equation}
where the Markov generator $L_m$ is given by
\begin{equation*}
(L_mv)(x)=\big\langle A_mx-P_m f(x),Dv(x)\big\rangle
-\half\Tr\big(P_m\g(x)Q\g^*(x)P_m D^2v(x)\big),
\quad v\in\mathcal{C}^2(H,\R),\ x\in H.
\end{equation*}
The proof of Theorem \ref{main} relies heavily on estimates of the derivatives $Du_m$ and $D^2u_m$ of the form: for some $\alpha>0$ we have
\begin{align}
\label{Lemmaux}
\sup_{x\in H}\|A^{\lambda}Du_m(t,x)\|&\leq Ct^{-\lambda}|\varphi|_{\Cb^1},\quad t\in(0,T],\ \lambda\in[0,\alpha),\\
\label{Lemmauxx}
\sup_{x\in H}\|A^{\lambda}D^2u_m(t,x)A^{\rho}\|_{\LB}&\leq Ct^{-(\rho+\lambda)}|\varphi|_{\Cb^2} ,\quad t\in(0,T],\ \lambda,\rho\in[0,\alpha),\ \lambda+\rho<1.
\end{align}
In the case of colored noise it turns out that we need $\alpha\geq(1+\beta)/2$ to obtain convergence of the right rate. So far, to our knowledge, there is no satisfactory result in this direction for multiplicative noise. But for additive colored noise the situation is much simpler and the estimates hold with $\alpha=1$, see Lemma 3.3 in \cite{WangGan}. For the white noise case the estimates are stated as Lemma 4.4 and Lemma 4.5 in \cite{Debussche} with $\alpha=\tfrac12$. Thus, in case \textbf{A} we have $\beta\in[\tfrac12,1]$, \eqref{Lemmaux} and \eqref{Lemmauxx} with $\alpha=1$, and in case \textbf{B} we have $\beta\in(0,\tfrac12)$ and \eqref{Lemmaux} and \eqref{Lemmauxx} with $\alpha=\tfrac12$.

Since we use the operator $A$ in \eqref{Lemmaux} and \eqref{Lemmauxx} instead of the more natural choice $A_m$, we outline the proofs. We use that $Du(t,x)\cdot\phi=\E[D\varphi(X_m^x(t))\cdot\eta_m^{\phi,x}(t)]$, where $\eta_m^{\phi,x}$ solves
\begin{equation*}
\begin{split}
\eta_m^{\phi,x}(t)&=S_m(t)P_m\phi+\int_0^tS_m(t-s)P_m Df(X_m^x(s))\cdot\eta_m^{\phi,x}(s)\diff{s}\\
&\quad+\int_0^tS_m(t-s)P_m\big(Dg(X_m^x(s)) \cdot\eta_m^{\phi,x}(s)\big)\diff{W(s)}.
\end{split}
\end{equation*}
In the proofs of Lemma 3.3 in \cite{WangGan} for the case \textbf{A} with $\alpha=1$ and Lemma 4.4 in \cite{Debussche} for the case \textbf{B} with $\alpha=\tfrac12$ it is proved that
\begin{equation}
\label{etaest}
\Big(\sup_{x\in H}\E\|\eta_m^{\phi,x}(t)\|^2\Big)^{\frac12}\lesssim t^{-\lambda}\|A_m^{-\lambda}P_m\phi\|,\quad t\in(0,T],\ \lambda\in [0,\alpha).
\end{equation}
Therefore,
\begin{equation*}
\begin{split}
\langle A^\lambda Du_m(t,x),\psi\rangle&=\langle  Du_m(t,x),A^\lambda\psi\rangle =\E[D\varphi(X_m^x(t))\cdot\eta_m^{A^\lambda\psi,x}(t)]\leq |\varphi|_{\Cb^1}\big(\E\|\eta_m^{A^\lambda\psi,x}(t)\|^2\big)^\frac12\\
&\lesssim |\varphi|_{\Cb^1} t^{-\lambda}\|A_m^{-\lambda} P_mA^{\lambda}\psi\|=|\varphi|_{\Cb^1} t^{-\lambda}\|P_m\psi\| \leq|G|_{\Cb^1} t^{-\lambda}\|\psi\|,
\end{split}
\end{equation*}
implying \eqref{Lemmaux}.

For \eqref{Lemmauxx} we notice that
\begin{equation}
\label{D2um}
D^2u_m(t,x)\cdot(\phi,\psi) =\E[D^2\varphi(X_m^x(t))\cdot(\eta_m^{\phi,x}(t),\eta_m^{\psi,x}(t)) +D\varphi(X_m^x(t))\cdot\zeta_m^{\phi,\psi,x}(t)],
\end{equation}
where $\zeta_m^{\phi,\psi,x}$ is the solution of
\begin{equation*}
\begin{split}
\zeta_m^{\phi,\psi,x}(t)& =\int_0^tS_m(t-s)P_m\big(D^2f(X_m^x(s)) \cdot(\eta_m^{\phi,x}(s),\eta_m^{\psi,x}(s)) +Df(X_m^x(s))\cdot\zeta_m^{\phi,\psi,x}(s)\big)\diff{s}\\
&\ +\int_0^tS_m(t-s)P_m\big(D^2g(X_m^x(s)) \cdot(\eta_m^{\phi,x}(s),\eta_m^{\psi,x}(s)) +Dg(X_m^x(s))\cdot\zeta_m^{\phi,\psi,x}(s)\big)\diff{W(s)}.
\end{split}
\end{equation*}
In the proofs of Lemma 3.3 in \cite{WangGan} for the case \textbf{A} with $\alpha=1$ and Lemma 4.5 in \cite{Debussche} for the case \textbf{B} with $\alpha=\tfrac12$ it is shown that
\begin{equation}
\label{zetaest}
\Big(\sup_{t\in[0,T]}\sup_{x\in H}\E\|\zeta_m^{\phi,\psi,x}(t)\|^2\Big)^\frac12\lesssim \|A_m^{-\rho}P_m\phi\|\|A_m^{-\lambda}P_m\psi\|, \quad \lambda,\rho\in[0,\alpha),\ \lambda+\rho<1.
\end{equation}
Since $D^2u_m\cdot(\phi,\psi)=\langle D^2u_m\phi,\psi\rangle$ and by \eqref{D2um} and the Cauchy-Schwarz inequality
\begin{equation*}
\begin{split}
&\langle  A^\lambda D^2u_m(t,x)A^\rho\phi,\psi\rangle= \langle D^2u_m(t,x)A^\rho\phi,A^\lambda\psi\rangle\\ &\qquad=\E\big[D^2\varphi(X_m^x(t))\cdot(\eta_m^{A^\lambda\psi,x}(t), \eta_m^{A^\rho\phi,x}(t))+D\varphi(X_m^x(t))\cdot \zeta_m^{A^\lambda\psi,A^\rho\phi,x}(t)\big]\\
&\qquad\leq|\varphi|_{\Cb^2} \big(\E\|\eta_m^{A^\lambda\psi,x}(t)\|^2\big)^\frac12 \big(\E\|\eta_m^{A^\rho\psi,x}(t)\|^2\big)^\frac12+|\varphi|_{\Cb^1} \big(\E\|\zeta_m^{A^\lambda\psi,A^\rho\phi,x}(t)\|^2\big)^\frac12.
\end{split}
\end{equation*}
Applying \eqref{etaest} and \eqref{zetaest} yields
\begin{equation*}
\langle  A^\lambda D^2u_m(t,x)A^\rho\phi,\psi\rangle
\lesssim(|\varphi|_{\Cb^2}t^{-\lambda-\rho}
+|\varphi|_{\Cb^1})\|A_m^{-\lambda}P_mA^\lambda\phi\| \|A_m^{-\rho}P_mA^\rho\psi\|
\lesssim t^{-\lambda-\rho}\|\phi\|\|\psi\|.
\end{equation*}
This implies \eqref{Lemmauxx}.

\section{Proof of Theorem \ref{main}} \label{sec:main} The error splits into several terms, some of which are common to Assumptions \textbf{A} and \textbf{B}. We first present the proof under Assumption \textbf{A}. When doing so we write it as if the noise were multiplicative, i.e., with the operator $g$ included. This will simplify the presentation of the additive case, Assumption \textbf{B}.

\subsection{Assumption \textbf{A}} \label{subsec:colored}
For an $\mathcal{F}_T$-measurable, $H_m$-valued random variable $\xi$, the law of iterated expectation and Proposition 1.12 in \cite{DaP} yields
\begin{equation}
\label{ItCondExp}
\E[\varphi(\xi)]=\E[\E[\varphi(\xi)|\mathcal{F}_T]] =\E[\E[\varphi(X^\xi(0))|\mathcal{F}_T]]=\E[u_m(0,\xi)].
\end{equation}
In particular, $\E[\varphi(X_h(T))]=\E[u_m(0,X_h(T))]=\E[u_m(0,P_mX_h(T))]$. We also denote $X_m(t)=X_m^{X_0}(t)$, so that $\E[\varphi(X_m(T))]=u_m(T,X_0)$, where $X_0$ is deterministic. Thus, the weak error splits as:
\begin{equation*}
\begin{split}
&\E[\varphi(X(T))-\varphi(X_h(T))]\\
&\quad=\E[\varphi(X(T))-\varphi(X_m(T))] +\E[\varphi(X_m(T))-\varphi(P_mX_h(T))]+\E[\varphi(P_mX_h(T))-\varphi(X_h(T))]\\
&\quad=\E[\varphi(X(T))-\varphi(X_m(T))]+u_m(T,X_0)-u_m(T,X_h(0))\\
&\qquad+\E[u_m(T,X_h(0))-u_m(0,X_h(T))]+\E[\varphi(P_mX_h(T))-\varphi(X_h(T))]\\
&\quad=e_1^m(T)+e_2^m(T)+e_3^m(T)+e_4^m(T).
\end{split}
\end{equation*}
The parameters $h$ and $m$ are coupled by letting
\begin{align}
\label{mh8}
\lambda_m=h^{-8}.
\end{align}
For the first term we use the rate of strong convergence $X_m(T)\rightarrow X(T)$, see \cite[Chapter 3]{KruseThesis}. We get using \eqref{mh8}
\begin{align*}
e_1^m(T)\leq|\varphi|_{\Cb^1}\|X(T)-X_m(T)\|_{L_2(\Omega,H)}\lesssim \lambda_{m+1}^{-\beta}=h^{8\beta}.
\end{align*}
The second term $e_2^m(T)$ is also easy. The computations are the same under both of our assumptions. Using the Cauchy-Schwarz inequality, estimate \eqref{Lemmaux} with $0\leq\lambda=\beta-\epsilon<\alpha$ where $\alpha=1$ or $\alpha=\tfrac12$, and the error estimate \eqref{Phleq}, we obtain for small $\epsilon>0$
\begin{equation*}
\begin{split}
e_2^m(T)&=u_m(T,X_0)-u_m(T,P_hX_0) =\int_0^1\frac{\textrm{d}}{\textrm{d}s}u_m(T,P_hX_0+s (I-P_h)X_0)\diff{s}\\
&=\int_0^1\Big\langle A^{\beta-\epsilon}Du_m (T,P_hX_0+s (I-P_h)X_0),A^{-\beta+\epsilon}(I-P_h)X_0\Big\rangle\diff{s}\\
&\leq\int_0^1\|A^{\beta-\epsilon}Du_m (T,P_hX_0+s (I-P_h)X_0)\| \|A^{-\beta+\epsilon}(I-P_h)\|_{\LB}\|X_0\|\diff{s}\\
&\lesssim h^{2\beta-2\epsilon}T^{-\beta+\epsilon}|\varphi|_{\Cb^1}\|X_0\|\lesssim h^{2\beta-2\epsilon},\quad \textrm{uniformly in }m.\\
\end{split}
\end{equation*}
Here we used that $\|A^{-\beta+\epsilon}(I-P_h)\|_{\LB} =\|(A^{-\beta+\epsilon}(I-P_h))^* \|_{\LB}=\|(I-P_h)A^{-\beta+\epsilon}\|_{\LB}\lesssim h^{2\beta-2\epsilon}$.

We now turn to the third error term $e_3^m(T)$. For this we need the Markov generator $L_h$ of the finite element solution $X_h$. It is given by
\begin{equation*}
(L_hv)(x)=\big\langle A_hx-P_hf(x),Dv(x)\big\rangle-\half\Tr\big(P_h\g(x)Q\g^*(x)P_hD^2v(x)\big),
\quad v\in\mathcal{C}^2(H,\R),\ x\in S_h.
\end{equation*}
It\^{o}'s formula and the Kolmogorov equation \eqref{kolmo} give that
\begin{equation*}
\begin{split}
e_3^m(T)&=-\E[u_m(T-t,X_h(t))-u_m(T-0,X_h(0))]\Big|_{t=T}\\
&=-\E\Big[\int_0^T\Big(\dot{u}_m(T-t,X_h(t))+L_hu_m(T-t,X_h(t))\Big)\diff{t}\Big]\\
&=\E\int_0^T(L_m-L_h)u_m(T-t,X_h(t))\diff{t}.\\
\end{split}
\end{equation*}
The error $e_3^m(T)$ now naturally divides into three terms:
\begin{equation*}
\begin{split}
|e_3^m(T)|&\leq\Big|\E\int_0^T\Big\langle(A_m-A_h)X_h(t),Du_m (T-t,X_h(t))\Big\rangle \diff{t}\Big|\\
&\quad+\Big|\E\int_0^T\Big\langle(P_m-P_h)f(X_h(t)),Du_m (T-t,X_h(t))\Big\rangle \diff{t}\Big|\\
&\quad+\Big|\half\E\int_0^T\Tr \Big\{\Big[P_m\g(X_h(t))Q\g^*(X_h(t))P_m-P_h\g(X_h(t))Q\g^*(X_h(t))P_h\Big]\\
&\qquad\times D^2u_m (T-t,X_h(t))\Big\}\diff{t}\Big|\\
&=I+J+K.\\
\end{split}
\end{equation*}

The Ritz projector $R_h$ can be expressed as $R_h=A_h^{-1}P_hA$. Observing this we can write
\begin{equation*}
\begin{split}
\langle (A_m-A_h)X_h,Du_m\rangle&=\langle (A_mP_h-P_hA_h)X_h,Du_m\rangle=\langle X_h,(P_hA_m-A_hP_h)Du_m\rangle\\
&=\langle X_h,A_hP_h(A_h^{-1}P_hA_m-I)Du_m\rangle=\langle X_h,A_hP_h(A_h^{-1}P_hAP_m-I)Du_m\rangle\\
&=\langle X_h,A_hP_h(R_h-I)P_mDu_m\rangle+\langle X_h,A_hP_h(P_m-I)Du_m\rangle.
\end{split}
\end{equation*}
This enables us to rewrite the term $I$ so that we can apply the error estimates \eqref{Rhleq} and \eqref{Pmleq} for $R_h$ and $P_m$, respectively. We substitute for $X_h$ the mild equation \eqref{SPDEapprox} and treat the terms separately and estimate
\begin{equation*}
\begin{split}
I&\leq\Big|\E\int_0^T\Big\langle S_h(t)P_hX_0,A_hP_h(R_h-I)P_mDu_m(T-t,X_h(t))\Big\rangle \diff{t}\Big|\\
&\quad+\Big|\E\int_0^T\Big\langle \int_0^tS_h(t-s)P_hf(X_h(s))\diff{s},A_hP_h(R_h-I)P_mDu_m (T-t,X_h(t))\Big\rangle \diff{t}\Big|\\
&\quad+\Big|\E\int_0^T\Big\langle \int_0^tS_h(t-s)P_h\g(X_h(s))\diff{W(s)}, A_hP_h(R_h-I)P_mDu_m (T-t,X_h(t))\Big\rangle\diff{t}\Big|\\
&\quad+\Big|\E\int_0^T\Big\langle A_hX_h, (P_m-I)Du_m (T-t,X_h(t))\Big\rangle\diff{t}\Big|\\
&=I_1^h+I_2^h+I_3^h+I^m.\\
\end{split}
\end{equation*}
For the terms $I_1^h$, $I_2^h$ and $I_3^h$ we treat Assumptions \textbf{A} and \textbf{B} separately and start with \textbf{A}; \textbf{B} is postponed to the next subsection. Let $\epsilon>0$ be small. Using \eqref{Rhleq}, \eqref{leq3}, \eqref{AEhleq}, and \eqref{Lemmaux} yields
\begin{equation*}
\begin{split}
I_1^h&=\Big|\E\int_0^T\Big\langle A_h^{1-\epsilon}S_h(t)P_hX_0,\\
&\qquad A_h^\epsilon P_h(R_h-I)A^{-\max(\frac12,\beta-\epsilon)}P_mA^{\max(\frac12,\beta-\epsilon)}Du_m (T-t,X_h(t))\Big\rangle \diff{t}\Big|\\
&\leq\E\int_0^T\|A_h^{1-\epsilon}S_h(t)P_h\|_{\LB}\|X_0\| \|A_h^{\epsilon}P_h(R_h-I)A^{-\max(\frac12,\beta-\epsilon)}\|_{\LB} \|P_m\|_{\LB}\\
&\qquad\times\sup_{x\in H}\|A^{\max(\frac12,\beta-\epsilon)} Du_m (T-t,x)\|\diff{t}\\
&\lesssim h^{\max(1,2\beta-2\epsilon)-2\epsilon}\int_0^T t^{-1+\epsilon}(T-t)^{-\max(\frac12,\beta-\epsilon)}\diff{t}\,|\varphi|_{\Cb^1}\|X_0\|\lesssim h^{2\beta-4\epsilon}.\\
\end{split}
\end{equation*}
The term $I_2^h$ is easily estimated as follows:
\begin{equation*}
\begin{split}
I_2^h&=\Big|\E\int_0^T\Big\langle \int_0^tA_h^{1-\epsilon}S_h(t-s)P_hf(X_h(s))\diff{s},\\
&\qquad A_h^{\epsilon}P_h(R_h-I)A^{-\max(\frac12,\beta-\epsilon)}P_mA^{\max(\frac12,\beta-\epsilon)}Du_m (T-t,X_h(t))\Big\rangle \diff{t}\Big|\\
&\leq\int_0^T\int_0^t\|A_h^{1-\epsilon}S_h(t-s)P_h\|_{\LB} \big(\E\|f(X_h(s))\|^2\big)^\frac12\\
&\quad\times\|A_h^{\epsilon}P_h(R_h-I)A^{-\max(\frac12,\beta-\epsilon)}\|_{\LB} \|P_m\|_{\LB}\,\sup_{x\in H}\|A^{\max(\frac12,\beta-\epsilon)}Du_m (T-t,x)\|\diff{s}\diff{t}.
\end{split}
\end{equation*}
Using \eqref{Rhleq}, \eqref{leq3}, \eqref{AEhleq}, \eqref{Lemmaux}, and \eqref{Bndfg} yields
\begin{equation*}
I_2^h\lesssim h^{\max(1,2\beta-2\epsilon)-2\epsilon} \int_0^T\int_0^t(T-t)^{-\max(\frac12,\beta-\epsilon)} (t-s)^{-1+\epsilon}\diff{s}\diff{t}\lesssim h^{2\beta-4\epsilon}.\\
\end{equation*}

For $I_3^h$ we use the Malliavin integration by parts formula from Lemma \ref{IBP} together with the chain rule \eqref{Chain2} to obtain the error representation
\begin{equation}
\label{calc1}
\begin{split}
I_3^h&= \Big|\E\int_0^T\Big\langle \int_0^tS_h(t-s)P_h\g(X_h(s))\diff{W(s)},\,A_hP_h(R_h-I)P_mDu_m (T-t,X_h(t))\Big\rangle\diff{t}\Big|\\
&=\Big|\E\int_0^T\int_0^t\Big\langle S_h(t-s)P_h\g(X_h(s)),\\
&\qquad A_hP_h(R_h-I)P_mD^2u_m (T-t,X_h(t))P_mD_sX_h(t) \Big\rangle_{\LB_2^0}\diff{s}\diff{t}\Big|.
\end{split}
\end{equation}
Distributing powers of $A$ and $A_h$ carefully and recalling $g(x)=I$, we write
\begin{equation*}
\begin{split}
&\big\langle S_hP_h,A_hP_h(R_h-I)P_mD^2u_m\,P_mD_sX_h\big\rangle_{\LB_2^0}=\big\langle A_h^{\frac{1+\beta}2-\epsilon}S_hA_h^{\frac{1-\beta}2} A_h^{\frac{\beta-1}2}P_h,\\
&\qquad A_h^{\frac{1-\beta}2+\epsilon}P_h(R_h-I)A^{-\max(\frac12,\frac{1+\beta}2-\epsilon)} P_m A^{\max(\frac12,\frac{1+\beta}2-\epsilon)}D^2u_m \,A^{\frac{1-\beta}2}P_mA^{\frac{\beta-1}2}D_sX_h\big\rangle_{\LB_2^0}.
\end{split}
\end{equation*}
Using the Cauchy-Schwarz inequality for $\LB_2^0$ and \eqref{HolderInf} yields
\begin{equation*}
\begin{split}
I_3^h&\leq\E\int_0^T\int_0^t\|A_h^{1-\epsilon}S_h(t-s)P_h\|_{\LB} \|A_h^{\frac{\beta-1}2}P_h\|_{\LB_2^0} \|A_h^{\frac{1-\beta}2+\epsilon}P_h(R_h-I) A^{-\max(\frac12,\frac{1+\beta}2-\epsilon)}\|_{\LB}\|P_m\|_{\LB}\\
&\qquad\times \sup_{x\in H}\|A^{\max(\frac12,\frac{1+\beta}2-\epsilon)}D^2u_m (T-t,x)A^{\frac{1-\beta}2}\|_{\LB} \|A^{\frac{\beta-1}2}D_sX_h(t)\|_{\LB_2^0}\diff{s}\diff{t}.
\end{split}
\end{equation*}
We use \eqref{leq3} to get $\|A_h^{\frac{\beta-1}2}P_h\|_{\LB_2^0}\lesssim \|A^{\frac{\beta-1}2}\|_{\LB_2^0}$. The norm equivalence \eqref{eqnorm} and the fact that $D_s^uX_h(t)\in V_h$, $\mathbf{P}$-a.s., for every $u\in U_0$ yields
\begin{equation*}
\|A^{\frac{\beta-1}2}D_sX_h(t)\|_{\LB_2^0}\lesssim \|A_h^{\frac{\beta-1}2}D_sX_h(t)\|_{\LB_2^0}.
\end{equation*}
The analyticity of the semigroup \eqref{AEhleq}, the error estimate \eqref{Rhleq} together with \eqref{leq3}, the gradient estimate \eqref{Lemmauxx}, Tonelli's theorem and the Cauchy-Schwarz inequality now imply that
\begin{equation*}
\begin{split}
I_3^h&\lesssim h^{\max(\frac12,2\beta-2\epsilon)-2\epsilon}|\varphi|_{C_\mathrm{b}^2} \|A^{\frac{\beta-1}2}\|_{\LB_2^0}\\
&\qquad\times\int_0^T\int_0^t \big(\E\|A_h^{\frac{\beta-1}2}D_sX_h(t)\|_{\LB_2^0}^2\big)^{\frac12} (T-t)^{-\max(\frac12,\frac{1+\beta}2-\epsilon)-\frac{1-\beta}2} (t-s)^{-1+\epsilon}\diff{s}\diff{t}.
\end{split}
\end{equation*}
Applying Lemma \ref{LemmaDsXh} we finally get
\begin{equation*}
I_3^h\lesssim h^{\max(\frac12,2\beta-2\epsilon)-2\epsilon}\int_0^T\int_0^t
(T-t)^{-\max(\frac12,\frac{1+\beta}2-\epsilon)-\frac{1-\beta}2} (t-s)^{-1+\epsilon}\diff{s}\diff{t}\lesssim h^{2\beta-4\epsilon}.
\end{equation*}
The term $I^m$ has a common treatment for Assumption \textbf{A} and \textbf{B}. Here we use the inverse estimate $\|A_hP_h\|\lesssim h^{-2}$. Using also \eqref{Pmleq}, \eqref{Lemmaux}, the Cauchy-Schwarz inequality, \eqref{moment}, and \eqref{mh8} yields
\begin{equation*}
\begin{split}
I^m&\leq\E\int_0^T\|A_hP_h\|_\LB\|X_h(t)\| \|(P_m-I)A^{-\frac12+\epsilon}\|_{\LB} \sup_{x\in H}\|A^{\frac12-\epsilon}Du_m (T-t,x)\|\diff{t}\\
& \lesssim \lambda_{m+1}^{-\frac12+\epsilon}h^{-2}\Big(\sup_{t\in[0,T]} \E\|X_h(t)\|^2\Big)^\frac12\int_0^T (T-t)^{-\frac12+\epsilon}\diff{t} \,\lesssim\,h^{8(\frac12-\epsilon)-2}\,\lesssim\, h^{2-8\epsilon}.
\end{split}
\end{equation*}
Summing up we see that $I\lesssim h^{2\beta-8\epsilon}$.

The term $J$ is considered next. Writing $P_m-P_h=(P_m-I)+(I-P_h)$ we get the natural decomposition $J\leq J^m+J^h$. Using the Cauchy-Schwarz inequality, \eqref{Lemmaux}, and \eqref{Bndfg}, yields for $i\in\{h,m\}$
\begin{equation*}
\begin{split}
J^i&= \Big|\E\int_0^T\Big\langle(I-P_i)Du_m (T-t,P_mX_h(t)),f(X_h(t))\Big\rangle \diff{s}\Big|\\
&\leq \int_0^T\|(I-P_i)A^{-\beta+\epsilon}\|_{\LB}\,\sup_{x\in H^m}\|A^{\beta-\epsilon}Du_m (T-t,x)\|\,\big(\E\|f(X_h(t))\|^2\big)^\frac12\diff{t}\\
&\lesssim \|(I-P_i)A^{-\beta+\epsilon}\|_{\LB}|\varphi|_{\Cb^1}\int_0^T (T-t)^{-\beta+\epsilon}\diff{t}.
\end{split}
\end{equation*}
By \eqref{Phleq}, \eqref{Pmleq}, and \eqref{mh8} we have $J^h\lesssim h^{2\beta-2\epsilon}$ and $J^m\lesssim \lambda_{m+1}^{-\beta+\epsilon}=h^{8\beta-8\epsilon}$.

For $K$ we write
\begin{equation*}
\begin{split}
&P_m\g Q\g^*P_m-P_h\g Q\g^*P_h\\
&\quad=P_h\g Q\g^*(I-P_h)+(I-P_h)\g Q\g^*P_m+(P_m+P_h)\g Q\g^*(P_m-I),
\end{split}
\end{equation*}
and hence we get the following decomposition:
\begin{equation*}
\begin{split}
2K&=\Big|\E\int_0^T\Tr\Big(\big[P_m\g(X_h(t)) Q\g^*(X_h(t))P_m -P_h\g(X_h(t)) Q\g^*(X_h(t))P_h\big]\\
&\qquad\times D^2u_m (T-t,P_mX_h(t))\Big)\diff{t}\Big|\\
&\leq\Big|\E\int_0^T\Tr\Big(P_h\g(X_h(t)) Q\g^*(X_h(t))(I-P_h)D^2u_m(T-t,P_mX_h(t))\Big)\diff{t}\Big|\\
 &\quad+\Big|\E\int_0^T\Tr\Big((I-P_h)\g(X_h(t)) Q\g^*(X_h(t))P_mD^2u_m(T-t,P_mX_h(t))\Big)\diff{t}\Big|\\
&\quad+\Big|\E\int_0^T\Tr\Big((P_m+P_h)\g(X_h(t)) Q\g^*(X_h(t))(P_m-I)D^2u_m(T-t,P_mX_h(t))\Big)\diff{t}\Big|\\
&=K_1^h+K_2^h+K^m.
\end{split}
\end{equation*}
Assumption \textbf{A} is treated first; \textbf{B} is postponed. By \eqref{TrST}, \eqref{TrL1}, and \eqref{HolderInf}, we have
\begin{equation*}
\begin{split}
&\Tr(P_hQ(I-P_h)D^2u_m)\\
&\quad=\Tr(P_hQ(I-P_h)D^2u_m A^{\frac{1-\beta}2}A^{\frac{\beta-1}2}) =\Tr(A^{\frac{\beta-1}2}P_hQ(I-P_h)D^2u_m A^{\frac{1-\beta}2})\\ &\quad=\Tr(A^{\frac{\beta-1}2}P_hA^{\frac{1-\beta}2}A^{\frac{\beta-1}2} QA^{\frac{\beta-1}2}A^{\frac{1-\beta}2}(I-P_h) A^{-\frac{1+\beta}2+\epsilon}A^{\frac{1+\beta}2-\epsilon}D^2u_m A^{\frac{1-\beta}2})\\
&\quad\leq\|A^{\frac{\beta-1}2}P_hA^{\frac{1-\beta}2}\|_{\LB} \|A^{\frac{\beta-1}2}\|_{\LB_2^0}^2 \|A^{\frac{1-\beta}2}(I-P_h) A^{-\frac{1+\beta}2+\epsilon}\|_{\LB} \|A^{\frac{1+\beta}2-\epsilon}D^2u_mA^{\frac{1-\beta}2}\|_{\LB},
\end{split}
\end{equation*}
where we used the fact that, again by \eqref{TrL1},
\begin{equation*}
\|A^{\frac{\beta-1}2}QA^{\frac{\beta-1}2}\|_{\LB_1}= \Tr((A^{\frac{\beta-1}2}Q^{\frac12})(A^{\frac{\beta-1}2}Q^{\frac12})^*) =\|A^{\frac{\beta-1}2}Q^{\frac12}\|_{\LB_2}^2 =\|A^{\frac{\beta-1}2}\|_{\LB_2^0}^2.
\end{equation*}

By \eqref{eqnorm} and \eqref{leq3}, $\|A^{\frac{\beta-1}2}P_hA^{\frac{1-\beta}2}\|_{\LB}\lesssim \|A_h^{\frac{\beta-1}2}P_hA^{\frac{1-\beta}2}\|_{\LB} \lesssim \|A^{\frac{\beta-1}2}A^{\frac{1-\beta}2}\|_{\LB}=1$. Using \eqref{Phleq}, and \eqref{Lemmauxx} now gives us
\begin{equation*}
K_1^h\lesssim h^{2\beta-2\epsilon}\|A^{\frac{\beta-1}2}\|_{\LB_2^0}^2\; |\varphi|_{\mathcal{C}_{\mathrm{b}}^2} \int_0^T(T-t)^{-1+\epsilon}\diff{t}\lesssim h^{2\beta-2\epsilon}.\\
\end{equation*}
For $K_2^h$ we compute similarly
\begin{equation*}
\begin{split}
\Tr((I-P_h)QP_mD^2u)& =\Tr(A^{-\frac{1+\beta}2+\epsilon}(I-P_h)A^{\frac{1-\beta}2} A^{\frac{\beta-1}2}QA^{\frac{\beta-1}2}A^{\frac{1-\beta}2} D^2u_mA^{\frac{1+\beta}2-\epsilon})\\
&\leq\|A^{-\frac{1+\beta}2+\epsilon}(I-P_h)A^{\frac{1-\beta}2}\|_{\LB} \|A^{\frac{\beta-1}2}\|_{\LB_2^0}^2 \|A^{\frac{1-\beta}2}D^2u_mA^{\frac{1+\beta}2-\epsilon}\|_{\LB},
\end{split}
\end{equation*}
where
\begin{equation*}
\|A^{-\frac{1+\beta}2+\epsilon}(I-P_h)A^{\frac{1-\beta}2}\|_{\LB}\leq \|(A^{-\frac{1+\beta}2+\epsilon}(I-P_h)A^{\frac{1-\beta}2})^*\|_{\LB}
= \|A^{\frac{1-\beta}2}(I-P_h)A^{-\frac{1+\beta}2+\epsilon}\|_{\LB},
\end{equation*}
so that \eqref{Phleq} applies. Hence,
\begin{equation*}
K_2^h\lesssim h^{2\beta-2\epsilon}\|A^{\frac{\beta-1}2}\|_{\LB_2^0}^2\; |\varphi|_{C_\mathrm{b}^2}\int_0^T(T-t)^{-1+\epsilon}\diff{t} \lesssim h^{2\beta-2\epsilon}.
\end{equation*}
The term $K^m$ is treated analogously as $K_1^h$. We obtain $K^m\lesssim \lambda_{m+1}^{-\beta+\epsilon}=h^{8\beta-8\epsilon}$ by \eqref{mh8}.

Finally, by the Lipschitz continuity of $\varphi$, the regularity of $X_h(T)$, and \eqref{mh8} we get
\begin{align*}
e_4^m(T)&\leq |\varphi|_{\Cb^1}\E\|(P_m-I)X_h(T)\|\leq \|(P_m-I)A^{-\frac\beta2}\|_\LB \big(\E\|A^{\frac\beta2}X_h(T)\|^2\big)^\frac12\\
 &\lesssim \lambda_{m+1}^{-\beta}\big(1+\|A^\frac\beta2X_0\|\big)\,\lesssim\, h^{8\beta}.
\end{align*}
We conclude that $|\E[\varphi(X(T))-\varphi(X_h(T))]|=O(h^{2\gamma})$ for any $\gamma<\beta$, which completes the proof under Assumption \textbf{A}.

\subsection{Assumption \textbf{B}} \label{subsec:white}
Now consider the case of Assumption \textbf{B}. The terms $I_1^h$, $I_2^h$, $I_3^h$, $K_1^h$, $K_2^h$, and $K^m$ need a special treatment. We recall that under Assumption \textbf{B}, $Q=I$, $\beta=\tfrac12$, $U_0=H$, and $\LB_2^0=\LB_2$. We now complete the proof with the remaining estimates. In addition to what was used for Assumption \textbf{A} we also use \eqref{Pmleq2} and \eqref{mh8} here.
\begin{equation*}
\begin{split}
I_1^h&=\Big|\E\int_0^T\Big\langle A_h^{1-\epsilon}S_h(t)P_hX_0, A_h^\epsilon P_h(R_h-I)A^{-\frac12}A^\epsilon P_mA^{\frac12-\epsilon}Du_m (T-t,X_h(t))\Big\rangle \diff{t}\Big|\\
&\leq\E\int_0^T\|A_h^{1-\epsilon}S_h(t)P_h\|_{\LB}\|X_0\| \|A_h^{\epsilon}P_h(R_h-I)A^{-\frac12}\|_{\LB} \|A_m^\epsilon\|_{\LB}\sup_{x\in H}\|A^{\frac12-\epsilon} Du_m (T-t,x)\|\diff{t}\\
&\lesssim h^{1-2\epsilon}\,\lambda_{m+1}^\epsilon\,\int_0^T t^{-1+\epsilon}(T-t)^{-\frac12+\epsilon}\diff{t}\,|\varphi|_{\Cb^1}\|X_0\|\lesssim h^{1-2\epsilon}\,h^{-8\epsilon}\,=h^{1-10\epsilon}.\\
\end{split}
\end{equation*}
Similarly:
\begin{equation*}
\begin{split}
I_2^h&=\Big|\E\int_0^T\Big\langle \int_0^tA_h^{1-\epsilon}S_h(t-s)P_hf(X_h(s))\diff{s},\\
&\qquad A_h^{\epsilon}P_h(R_h-I)A^{-\frac12}A^\epsilon P_mA^{\frac12-\epsilon}Du_m (T-t,X_h(t))\Big\rangle \diff{t}\Big|\\
&\leq\int_0^T\int_0^t\|A_h^{1-\epsilon}S_h(t-s)P_h\|_{\LB} \big(\E\|f(X_h(s))\|^2\big)^\frac12\\
&\quad\times\|A_h^{\epsilon}P_h(R_h-I)A^{-\frac12}\|_{\LB} \|A_m^\epsilon\|_{\LB}\,\sup_{x\in H}\|A^{\frac12-\epsilon}Du_m (T-t,x)\|\diff{s}\diff{t},
\end{split}
\end{equation*}
and
\begin{equation*}
I_2^h\lesssim h^{1-2\epsilon}\, \lambda_{m+1}^\epsilon\, \int_0^T\int_0^t(T-t)^{-\frac12+\epsilon} (t-s)^{-1+\epsilon}\diff{s}\diff{t}\,\lesssim\,  h^{1-12\epsilon}.\\
\end{equation*}
Using H\"{o}lder's inequality \eqref{Holder1inf} in \eqref{calc1} gives us
\begin{equation*}
\begin{split}
I_3^h&=\Big|\E\int_0^T\int_0^t\Big\langle S_h(t-s)P_h\g(X_h(s)),A_hP_h(R_h-I)A^{-\frac12}\\
&\qquad\times A^\epsilon P_m A^{\frac12-\epsilon}D^2u_m (T-t,P_mX_h(t))P_mD_sX_h(t) \Big\rangle_{\LB_2}\diff{s}\diff{t}\Big|\\
&\leq\E\int_0^T\int_0^t \|A_h^{1-3\epsilon}S_h(t-s)P_h\|_{\LB}\|\g(X_h(s))\|_{\LB} \|A_h^{3\epsilon}P_h(R_h-I)A^{-\frac12}\|_{\LB}\|A^{\epsilon}P_m\|_\LB\\
&\qquad\times\sup_{x\in H^m}\|A^{\frac12-\epsilon}D^2u_m (T-t,x)A^{\frac12-\epsilon}\|_{\LB} \|A^{-\frac12+\epsilon}A_h^{-2\epsilon}P_h\|_{\LB_1} \|A_h^{2\epsilon}D_sX_h(t)\|_{\LB}\diff{s}\diff{t}.
\end{split}
\end{equation*}
First, using \eqref{eqnorm} and \eqref{Awhite}, we have
\begin{equation*}
\|A^{-\frac12+\epsilon}A_h^{-2\epsilon}P_h\|_{\LB_1} \lesssim\|A^{2\epsilon}A_h^{-2\epsilon}P_h \|_{\LB}\|A^{-\frac12-\epsilon}\|_{\LB_1} \lesssim \|A^{-\frac12-\epsilon}\|_{\LB_1}.
\end{equation*}
Now we apply \eqref{Rhleq}, \eqref{AEhleq}, \eqref{Pmleq2}, \eqref{Lemmauxx} with $\rho=\lambda=\tfrac12-\epsilon<\alpha=\tfrac12$, to get
\begin{equation*}
\begin{split}
I_3^h&\lesssim h^{1-6\epsilon}\,\lambda_{m+1}^{\epsilon}\,|\varphi|_{\Cb^2}\int_0^T\int_0^t\big(\E\|\g(X_h(s))\|_{\LB}^2\big)^{\frac12} \big(\E\|A_h^{2\epsilon}D_sX_h(t)\|_{\LB}^2\big)^{\frac12}\\
&\qquad\times (T-t)^{-1+2\epsilon}(t-s)^{-1+3\epsilon}\diff{s}\diff{t}.
\end{split}
\end{equation*}
Finally, using Lemma \ref{Dsleq}, \eqref{Bndfg} and \eqref{mh8} finishes the estimate of $I_3^h$. Indeed,
\begin{equation*}
I_3^h\lesssim h^{1-6\epsilon}\,h^{-8\epsilon}\,|\varphi|_{C_\mathrm{b}^2}\int_0^T (T-t)^{-1+2\epsilon} (t-s)^{-1+\epsilon}\diff{s}\diff{t}\lesssim h^{1-14\epsilon}.
\end{equation*}

For $K_1^h$ we use H\"{o}lder's inequality \eqref{Holder}, \eqref{Phleq}, and \eqref{Lemmauxx} to get
\begin{equation*}
\begin{split}
2K_1^h&\leq\int_0^T\E\|A^{-\frac{1-\epsilon}{2}} P_h\g(X_h(t))\g^*(X_h(t))A^{-\epsilon}\|_{\LB_1} \|A^{\epsilon}(I-P_h)A^{-\frac{1-\epsilon}{2}}\|_{\LB}\\
&\qquad\times\sup_{x\in H}\|A^{\frac{1-\epsilon}{2}}D^2u_m (T-t,x)A^{\frac{1-\epsilon}{2}}\|_{\LB}\diff{t}\\
&\lesssim h^{1-3\epsilon}\sup_{t\in[0,T]}\E\|A^{-\frac{1-\epsilon}{2}} P_h\g(X_h(t))\g^*(X_h(t))\|_{\LB_{2/(2-3\epsilon)}} \|A^{-\epsilon}\|_{\LB_{2/3\epsilon}}|\varphi|_{\Cb^2} \int_0^T(T-t)^{-1+\epsilon}\diff{t}\\
&\lesssim h^{1-3\epsilon}\sup_{t\in[0,T]}\E\|\g(X_h(t))\|_{\LB}^2 \|A^{-\frac{1-\epsilon}{2}}\|_{\LB_{2/(2-3\epsilon)}} \|A^{-\epsilon}\|_{\LB_{2/3\epsilon}} |\varphi|_{\Cb^2},
\end{split}
\end{equation*}
where \eqref{Bndfg} applies.  We use \eqref{Awhite} in the remaining terms:
\begin{align*}
\|A^{-\epsilon}\|_{\LB_{2/3\epsilon}}^{^{3\epsilon/2}}& =\sum_{i\in\N}(\lambda_i^{-\epsilon})^{\frac2{3\epsilon}} =\sum_{i\in\N}\lambda_i^{-\frac23}=\Tr(A^{-\frac23})<\infty,\\
\|A^{-\frac{1-\epsilon}2}\|_{\LB_{2/(2-3\epsilon)}}^{(2-3\epsilon)/2} &=\sum_{i\in\N}(\lambda_i^{-\frac{1-\epsilon}2})^{\frac2{2-3\epsilon}}= \sum_{i\in\N}\lambda_i^{-\frac{1-\epsilon}{2-3\epsilon}} =\Tr\Big(A^{-\frac12\big(\frac{2-2\epsilon}{2-3\epsilon}\big)}\Big)<\infty.
\end{align*}
The terms $K_2^h$ and $K^m$ admit the same treatment, so that $K_2^h\lesssim h^{1-3\epsilon}$ and $K^m\lesssim \lambda_{m+1}^{-\frac12+\frac{3\epsilon}2}=h^{4h-12\epsilon}$.  We thus have $|\E[\varphi(X(T))-\varphi(X_h(T))]|=\mathcal{O}(h^{2\gamma})$ for any $\gamma<\tfrac12$.

\providecommand{\bysame}{\leavevmode\hbox to3em{\hrulefill}\thinspace}
\providecommand{\MR}{\relax\ifhmode\unskip\space\fi MR }
\providecommand{\MRhref}[2]{%
  \href{http://www.ams.org/mathscinet-getitem?mr=#1}{#2}
}
\providecommand{\href}[2]{#2}

\end{document}